\documentclass[11pt]{amsart}
\usepackage{latexsym}
\usepackage{graphicx,subfigure}
\usepackage{array}
\usepackage{indentfirst}
\usepackage{fancyhdr}
\usepackage{epsf}
\usepackage{amsmath,bm}
\usepackage{amsthm}
\usepackage{booktabs}
\usepackage{amsfonts}
\usepackage{epsf}
\usepackage{multicol}
\usepackage{multirow}
\usepackage[square, comma, sort&compress, numbers]{natbib}
\usepackage{fancyhdr}
\usepackage{pxfonts}
\usepackage{caption}
\usepackage{epstopdf}
\usepackage{algorithm}
\usepackage{algpseudocode}
\usepackage{exscale}
\usepackage{relsize}
\usepackage{setspace}
\usepackage{geometry}
\usepackage{tikz}
\usetikzlibrary{arrows,shapes,chains,positioning}
\usetikzlibrary{decorations.markings}
\tikzstyle arrowstyle=[scale=1]
\tikzstyle directed=[postaction={decorate,decoration={markings,mark=at position .65 with {\arrow[arrowstyle]{stealth}}}}]
\tikzstyle reverse directed=[postaction={decorate,decoration={markings,mark=at position .65 with {\arrowreversed[arrowstyle]{stealth};}}}]

\newtheorem{Def}{Definition}
\newtheorem{Th}{Theorem}[section]
\newtheorem{Lm}{Lemma}[section]

\newtheorem{remark}{Remark}[section]

\newtheorem{theorem}{Theorem}[section]
\newtheorem{lemma}[theorem]{Lemma}

\theoremstyle{definition}
\usepackage[titletoc]{appendix}
\numberwithin{equation}{section}

\newcommand{\be}{\begin{eqnarray}}
\newcommand{\ee}{\end{eqnarray}}
 
\newcommand{\Om}{\Omega}
\newcommand{\dx}{\,dx}

\newcommand{\ds}{\,ds}

\newcommand{\Rmnum}[1]{\expandafter\@slowromancap\romannumeral #1@}

\newcommand{\PiCR}{\Pi_{\rm CR}}
\newcommand{\PiECR}{\Pi_{\rm ECR}}
\newcommand{\PiP}{\Pi_{\rm P_1}}
\newcommand{\VCR}{V_{\rm CR}}
\newcommand{\VECR}{V_{\rm ECR}}
\newcommand{\CRlam}{\lambda_{\rm CR} }
\newcommand{\ECRlam}{\lambda_{\rm ECR}}
\newcommand{\CRu}{u_{\rm CR} }
\newcommand{\ECRu}{u_{\rm ECR} }
\newcommand{\Mu}{u_{\rm M} }
\newcommand{\Mlam}{\lambda_{\rm M} }
\newcommand{\Poneu}{u_{\rm P_1} }
\newcommand{\Ponelam}{\lambda_{\rm P_1} }
\newcommand{\PoneAstu}{u_{\rm P_1^\ast} }
\newcommand{\PoneAstlam}{\lambda_{\rm P_1^\ast} }

\newcommand{\cE}{\mathcal{E}}

\newcommand{\cT}{\mathcal{T}}
\newcommand{\R}{\mathbb{R}}

\renewcommand{\arraystretch}{1.5}

\newcommand{\yemeifont}{\fontsize{9pt}{\baselineskip}\selectfont}
\begin{document}
\title{
Asymptotically exact a posteriori error  estimates of eigenvalues by the Crouzeix-Raviart element and enriched Crouzeix-Raviart element
}

\author {Jun Hu}
\address{LMAM and School of Mathematical Sciences, Peking University,
  Beijing 100871, P. R. China.  hujun@math.pku.edu.cn}

\author{Limin Ma}
\address{LMAM and School of Mathematical Sciences, Peking University,
  Beijing 100871, P. R. China. maliminpku@gmail.com}
\thanks{The authors were supported by  NSFC
projects 11625101, 91430213 and 11421101}
\maketitle

\begin{abstract}
Two asymptotically exact a posteriori error estimates are proposed for eigenvalues  by the nonconforming  Crouzeix--Raviart  and  enriched Crouzeix--Raviart  elements. The main challenge in the design of such error estimates comes from  the  nonconformity of the  finite element spaces used. Such nonconformity  causes two difficulties, the first one is the construction of  high accuracy gradient recovery algorithms, the second  one is  a computable high accuracy approximation of a consistency error term.  The first difficulty was solved  for both nonconforming elements in a previous paper. Two methods are proposed to solve the second difficulty in the present paper. In particular, this  allows the use of high accuracy gradient recovery techniques. Further,  a post-processing algorithm  is  designed by utilizing asymptotically exact a posteriori error estimates to construct  the weights of a combination of two approximate eigenvalues. This algorithm requires to solve only one eigenvalue problem and admits high accuracy eigenvalue approximations both theoretically and numerically.
 \vskip 15pt
 

 \vskip 15pt
 
\end{abstract}

\section{Introduction}
Asymptotically exact a posteriori error estimates are  widely used to improve accuracy of approximations. A posteriori error estimates were first proposed by Babu\v{s}ka and Rheinboldt in 1978 \cite{babuvska1978posteriori}. Since then, some important branches of a posteriori error estimates have been developed, such as residual type a posteriori error estimates \cite{ainsworth2011posteriori, becker2010convergent,  carstensen2007unifying,carstensen2007framework} 
and recovery type a posteriori error estimates \cite{babuvska1994validation,carstensen2004all,qun1993notes,yan2001gradient,zhang2005new,zienkiewicz1992superconvergent}. For eigenvalues of the Laplacian operator by the conforming linear element, asymptotically exact a posteriori error estimates were proposed and analyzed in \cite{Zhang2006Enhancing}. These error estimates play an important role in improving the accuracy of eigenvalues to a remarkable fourth order.
For that conforming element, the error in the eigenvalues can be decomposed into  the energy norm  of the error in the approximation of the eigenfunctions and a higher order term. Asymptotically exact a posteriori error estimates follow directly from this crucial fact and an application of high accuracy gradient recovery techniques, such as  polynomial preserving recovery techniques (PPR for short hereinafter) in \cite{guo2015gradient,Zhang2005A}, Zienkiewicz-Zhu superconvergence patch recovery techniques in \cite{Zienkiewicz1992The} and that superconvergent cluster recovery methods in \cite{huang2010superconvergent}. 

As for nonconforming elements, the error for the eigenvalues is composed of the energy norm of the error  in the approximation of  the corresponding eigenfunctions, an extra consistency error term and a higher order term. The nonconformity causes two major difficulties, the first one is the construction of  high accuracy gradient recovery algorithms, the second  one is  a computable high accuracy approximation of the consistency error term. A previous paper \cite{hu2018optimal} analyzed an optimal superconvergence result for both the nonconforming Crouzeix-Raviart (CR for short hereinafter)  element and the enriched Crouzeix-Raviart (ECR for short hereinafter) element. It offers a computable high accuracy approximation to the gradient of the eigenfunctions. The aforementioned consistency error term requires to approximate the eigenfunctions themselves with high accuracy, not the gradient of the eigenfunctions any more. The fact that there exist no such high accuracy function recovery techniques in literature causes the second difficulty for nonconforming elements. 

Two types of asymptotically exact a posteriori error estimates for the eigenvalues are designed for the nonconforming CR element and the ECR element. The main idea here is to turn this function recovery problem into a high accuracy gradient recovery problem. The first type of a posteriori error estimates employs a commuting interpolation of eigenfunctions, and the second type makes use of a conforming interpolation of eigenfunctions. Both types of asymptotically exact a posteriori error estimates require a high accuracy gradient recovery technique for the nonconforming CR element and the ECR element. The first design of asymptotically exact a posteriori error estimates is much easier to implement but requires a commuting interpolation, while the other one applies for more general nonconforming elements as long as the corresponding discrete space contains a conforming subspace. Although both error estimates achieve the same predicted accuracy, experiments indicate even higher accuracy for the first type of error estimates when the eigenfunctions are smooth enough.  While the second one admits much better experimental performance when the eigenfunctions are singular.

An additional technique for high precision eigenvalues is to combine two approximate eigenvalues by a weighted-average \cite{Hu2012A}. 
The accuracy of the resulted combined eigenvalues depends on the accuracy of the weights.
In \cite{Hu2012A}, two finite elements are employed to solve eigenvalue problems on two meshes with one element producing upper bounds of the eigenvalues and the other one producing lower bounds. The main idea there is to design approximate weights through these four resulted discrete eigenvalues.
This algorithm is observed to be quite efficient by experiments.

By use of the aforementioned asymptotically exact a posteriori error estimates, a new post-processing algorithm is proposed and analyzed to improve the accuracy of the approximate eigenvalues. Given lower bounds of the eigenvalues and the corresponding nonconforming approximate eigenfunctions, an application of the average-projection method \cite{Hu2015Constructing} to these eigenfunctions yields conforming approximate eigenfunctions. By \cite{Hu2015Constructing}, Rayleigh quotients of such conforming eigenfunctions are asymptotic upper bounds of the eigenvalues. The new algorithm combines the lower bounds and the upper bounds of the eigenvalues by a weighted average. The weights here are designed from the corresponding asymptotically exact a posteriori error estimates. The resulted combined eigenvalues are proved to admit higher accuracy both theoretically and experimentally. It needs to point out that only one discrete eigenvalue problem needs to be solved in this new algorithm.

The remaining paper is organized as follows. Section 2 presents second order elliptic eigenvalue problems and some notations. Section 3 establishes and analyzes asymptotically exact a posteriori error estimates for eigenvalues by the nonconforming CR element and the ECR element. Section 4 proposes two post-processing algorithms to approximate eigenvalues with high accuracy. Section 5 presents some numerical tests.

\section{Notations and Preliminaries}
\subsection{Notations}
We first introduce some basic notations. Given a nonnegative integer $k$ and a bounded domain $\Om\subset \mathbb{R}^2$ with Lipchitz boundary $\partial \Om$, let $H^k(\Om,\mathbb{R})$, $\parallel \cdot \parallel_{k,\Om}$ and $|\cdot |_{k,\Om}$ denote the usual Sobolev spaces, norm, and semi-norm, respectively. Denote the standard $L^2(\Om,\mathbb{R})$ inner product and $L^2(K,\mathbb{R})$ inner product by $(\cdot, \cdot)$ and $(\cdot, \cdot)_{0,K}$, respectively. Let $H_0^1(\Om,\mathbb{R}) = \{u\in H^1(\Om,\mathbb{R}): u|_{\partial \Om}=0\}$.

Suppose that $\Om\subset \mathbb{R}^2$ is a bounded polygonal domain covered exactly by a shape-regular partition $\cT_h$ into simplices. Let element $K$ have vertices $\bold{p}_i=(p_{i1},p_{i2}),1\leq i\leq 3$ oriented counterclockwise, and corresponding barycentric coordinates $\{\phi_i\}_{i=1}^3$. Denote $\{e_i\}_{i=1}^3$ the edges of element $K$, and $\{\bold{t}_i\}_{i=1}^3$ the unit tangent vectors with counterclockwise orientation. Denote the column vectors $\bold{e}_1=(1, 0)^T$ and 
 $\bold{e}_2=(0, 1)^T$.

Let $|K|$ denote the volume of element $K$ and $|e|$ the length of edge $e$. Let $h_K$ denote the diameter of element $K\in \cT_h$ and $h=\max_{K\in\cT_h}h_K$. Denote the set of all interior edges and boundary edges of $\cT_h$ by $\cE_h^i$ and $\cE_h^b$, respectively, and $\cE_h=\cE_h^i\cup \cE_h^b$. For any interior edge $e=K_e^1\cap K_e^2$, we denote the element with larger global label by $K_e^1$, the one with smaller global label by $K_e^2$. Let $\{\cdot\}$ and $[\cdot]$ be the average and jump of piecewise functions over edge $e$, namely
$$
\{v\}|_e := \frac{1}{2}(v|_{K_e^1}+v|_{K_e^2}),\qquad [v]|_e := v|_{K_e^1}-v|_{K_e^2}
$$
for any piecewise function $v$. For $K\subset\R^2,\ r\in \mathbb{Z}^+$, let $P_r(K)$ be the space of all polynomials of degree not greater than $r$ on $K$. Denote the second order derivatives $\frac{\partial^2 u}{\partial x_i\partial x_j}$ by $\partial_{x_ix_j} u$, $1\leq i, j\leq 2$, the piecewise gradient operator and the piecewise Hessian operator by $\nabla_h$ and $\nabla_h^2$, respectively.

Throughout the paper, a positive constant independent of the mesh size is denoted by $C$, which refers to different values at different places. 
\subsection{Nonconforming elements for eigenvalue problems}
We consider a model eigenvalue problem of finding : $(\lambda, u)\in \mathbb{R}\times V$ such that $\parallel u \parallel_{0,\Om}=1$  and
\be\label{variance}
a(u, v)=\lambda(u, v) \text{\quad for any }v\in V,
\ee
where $V:= H^1_0(\Om,\mathbb{R})$. The bilinear form
$
a(w, v):=\int_{\Om} \nabla  w\cdot \nabla v \dx
$
is symmetric, bounded, and coercive, namely for any $ w, v\in V$,
$$
a(w,v)=a(v,w),\quad |a(w,v)|\leq C  \parallel w\parallel_{1,\Om}\parallel v\parallel_{1,\Om},\quad \parallel v\parallel_{1,\Om}^2\leq C  a(v,v).
$$
The eigenvalue problem \eqref{variance} has a sequence of eigenvalues
$$0<\lambda_1\leq \lambda_2\leq \lambda_3\leq ...\nearrow +\infty,$$
and the corresponding eigenfunctions
$u_1, u_2, u_3,... ,$
with
$$(u_i, u_j)=\delta_{ij}\ \text{ with } \delta_{ij}=\begin{cases}
0 &\quad i\neq j\\
1 &\quad i=j
\end{cases}.$$

Let $V_h$ be a nonconforming finite element approximation of $V$ over $\cT_h$. The corresponding finite element approximation of \eqref{variance} is to find $(\lambda_h, u_h)\in \mathbb{R}\times V_h$ such that $\parallel u_h\parallel_{0,\Om}=1$ and
\be\label{discrete}
a_h(u_h,v_h)=\lambda_h(u_h, v_h)\quad \text{ for any }v_h\in V_h,
\ee
with the discrete bilinear form $a_h(w_h,v_h)$ defined elementwise as
$$
a_h(w_h,v_h):=\sum_{K\in\cT_h}\int_K \nabla_h w_h\cdot \nabla_h v_h\dx.
$$
Let $N=\text{dim }V_h$. Suppose that $ \parallel \cdot \parallel_h:=a_h(\cdot, \cdot)^{1/2}$ is a norm over the discrete space $V_h$, the discrete problem \eqref{discrete} admits a sequence of discrete eigenvalues
$$0<\lambda_{1,h}\leq \lambda_{2,h}\leq \lambda_{3,h}\leq ...\leq \lambda_{N,h},$$
and the corresponding eigenfunctions
$
u_{1,h}, u_{2,h},..., u_{N,h}
$
with $(u_{i,h}, u_{j,h})=\delta_{ij}.$

We consider the following two nonconforming elements: the CR element and the ECR element.

$\bullet$\quad  The CR element space over $\cT_h$ is defined in \cite{Crouzeix1973Conforming} by
\begin{equation*}
\begin{split}
\VCR:=&\big \{v\in L^2(\Om,\mathbb{R})\big|v|_K\in P_1(K)\text{ for any }  K\in\cT_h, \int_e [v]\ds =0\text{ for any }  e\in \cE_h^i,\\
&\int_e v\ds=0\text{ for any }  e\in \cE_h^b\big\}.
\end{split}
\end{equation*}
Moreover, we define the canonical interpolation operator $\PiCR: H^1_0(\Om,\mathbb{R})\rightarrow \VCR$ as follows:
\be\label{crinterpolation}
\int_e\PiCR v\ds=\int_e v\ds\quad \text{ for any } e\in \cE_h,\ v\in H^1_0(\Om,\mathbb{R}).
\ee
Denote the approximate eigenpair of \eqref{discrete} in the nonconforming space $\VCR$ by \\
$(\CRlam, \CRu)$ with $ \parallel \CRu\parallel_{0,\Om}=1$.

$\bullet$\quad  The ECR element space over $\cT_h$ is defined in \cite{Hu2014Lower} by
\begin{equation*}
\begin{split}
\VECR:=&\big \{v\in L^2(\Om,\mathbb{R})\big|v|_K\in \text{ECR}(K)\text{ for any }  K\in\cT_h, \int_e [v]\ds =0\text{ for any }  e\in \cE_h^i,\\
&\int_e v\ds=0\text{ for any }  e\in \cE_h^b\big\}.
\end{split}
\end{equation*}
with $\text{ECR}(K)=P_1(K)+\text{span}\big\{x_1^2+x_2^2\big\}$. Define the canonical interpolation operator $\PiECR: H^1_0(\Om,\mathbb{R})\rightarrow \VECR $ by
\begin{equation}\label{ecrinterpolation}
\int_e\PiECR v\ds= \int_e v\ds,\quad
\int_K \PiECR v\dx= \int_K v\dx\quad\forall e\in\cE_h,  K\in\cT_h.
\end{equation}
Denote the approximate eigenpair of \eqref{discrete} in the nonconforming space $\VECR$ by $(\ECRlam,\ECRu)$ with $ \parallel \ECRu\parallel_{0,\Om}=1$.

It follows from the theory of nonconforming approximations of eigenvalue problems in \cite{Hu2014Lower,rannacher1979nonconforming} that
\begin{equation}\label{CR:est}
|\lambda-\CRlam|+\parallel u- \CRu\parallel_{0,\Om} + h^s\parallel \nabla_h (u-\CRu)\parallel_{0,\Om}\leq C  h^{2s}\parallel u\parallel_{1+s,\Om},
\end{equation}
\begin{equation}\label{ECR:est}
|\lambda-\ECRlam|+\parallel u - \ECRu\parallel_{0,\Om} + h^s\parallel \nabla_h (u-\ECRu)\parallel_{0,\Om}\leq C  h^{2s}\parallel u\parallel_{1+s,\Om},
\end{equation}
provided that $u\in H^{1+s}(\Om,\mathbb{R})\cap H^1_0(\Om,\mathbb{R})$, $\ 0<s\leq 1$.

For the CR element and the ECR element, there holds the following commuting property of the canonical interpolations
\begin{equation}\label{commuting}
\begin{split}
\int_K \nabla(w - \PiCR w)\cdot \nabla v_h \dx &=0\quad\text{ for any } w\in V, v_h\in \VCR,\\
\int_K \nabla(w - \PiECR w)\cdot \nabla v_h \dx &=0\quad\text{ for any }w\in V, v_h\in \VECR,
\end{split}
\end{equation}
see \cite{Crouzeix1973Conforming,Hu2014Lower} for details. 

\subsection{Some Taylor expansions}
On each element $K$, denote the centroid of element $K$ by $\bold{M}_K=(M_1,M_2)$. Let 
$
A_K=\sum_{i,j =1, i\neq j}^3 \big ((p_{i1}-p_{j1})^2 - (p_{i2}-p_{j2})^2\big ),
$
$
B_K= \sum_{i=1}^3\big( 2p_{i1}p_{i2} -\sum_{j=1,j\neq i}^3 p_{i1}p_{j2}\big  ),
$ 
and $H_K=\sum_{i=1}^3 |e_i|^2 $. We introduce three short-hand notations
\begin{align*}\label{consDef}
\phi_{\rm ECR}^1(\bold{x})&=(x_1-M_1)^2-(x_2-M_2)^2,
&\phi_{\rm ECR}^2(\bold{x})&=(x_1-M_1)(x_2-M_2),\\
\phi_{\rm ECR}^3(\bold{x})&=2 - \frac{36}{H_K}\sum_{i=1}^2(x_i-M_i)^2.&
\end{align*}
Note that 
$
P_2(K) = P_1(K) + {\rm span}\{\phi_{\rm ECR}^1, \phi_{\rm ECR}^2, \phi_{\rm ECR}^3\}.
$ 

For any element $K$ and $H_h\in L^2(K, \mathbb{R}^{2\times 2})$, define 
{\small
\begin{equation}\label{Pdefine}
\resizebox{0.9\textwidth}{!}{ $
\begin{split}
P_{\rm P_1}^K(H_h)=&-\frac{1}{2}\sum_{i=1}^3 |e_i|^2\phi_{i+1}\phi_{i-1}\frac{\|\bold{t}_i^TH_h\bold{t}_i\|_{0,K}}{|K|^{\frac{1}{2}}},\\
P_{\rm CR}^K(H_h)=&\frac{ \|\bold{e}_1^TH_h\bold{e}_1 - \bold{e}_2^TH_h\bold{e}_2\|_{0, K}}{4|K|^{\frac{1}{2}}}(I-\PiECR)\phi_{\rm ECR}^1 + \frac{  \|\bold{e}_1^TH_h\bold{e}_2\|_{0, K}}{|K|^{\frac{1}{2}}} (I-\PiECR)\phi_{\rm ECR}^2\\
&-\big(
\frac{  A_K+H_K }{144|K|^{\frac{1}{2}}} \| \bold{e}_1^TH_h\bold{e}_1\|_{0, K} + \frac{ H_K-A_K }{144|K|^{\frac{1}{2}}} \|\bold{e}_2^TH_h\bold{e}_2\|_{0, K}
\\
& + \frac{ B_K}{36|K|^{\frac{1}{2}}}\|\bold{e}_1^TH_h\bold{e}_2\|_{0, K}
\big)  \phi_{\rm ECR}^3\\
P_{\rm ECR}^K(H_h)=&\frac{ \|\bold{e}_1^TH_h\bold{e}_1-\bold{e}_2^TH_h\bold{e}_2\|_{0, K}}{4|K|^{\frac{1}{2}}}(I-\PiECR)\phi_{\rm ECR}^1 +   \frac{\|\bold{e}_1^TH_h\bold{e}_2\|_{0, K}}{|K|^{\frac{1}{2}}} (I-\PiECR)\phi_{\rm ECR}^2
\end{split}
$}
\end{equation}
}
where $\phi_i$ is the corresponding barycentric coordinate to vertex $\bold{p}_i$. For any $H_h\in L^2(K, \mathbb{R}^{2\times 2})$, both $\frac{\|\bold{t}_i^TH_h\bold{t}_j\|_{0,K}}{|K|^{\frac{1}{2}}}$ and $\frac{ \| \bold{e}_i^TH_h\bold{e}_j\|_{0, K}}{|K|^{\frac{1}{2}}}$ are constant. Thus, all these  functions $P_{\rm P_1}^K(H_h)$, $P_{\rm CR}^K(H_h)$ and $P_{\rm ECR}^K(H_h)$ are polynomials of order two. Note that if $H_h\in P_0(K, \mathbb{R}^{2\times 2})$, as used in numerical computations,
$$
\frac{\|\bold{t}_i^TH_h\bold{t}_j\|_{0,K}}{|K|^{\frac{1}{2}}}=\bold{t}_i^T|H_h|\bold{t}_j,\quad \frac{ \| \bold{e}_i^TH_h\bold{e}_j\|_{0, K}}{|K|^{\frac{1}{2}}}=\bold{e}_i^T|H_h|\bold{e}_j.
$$

The following lemma lists the Taylor expansion of the canonical interpolation error for the conforming linear element, the CR element and the ECR element, respectively. See \cite{hu2019asymptotic,huang2008superconvergence} for more details.

\begin{Lm}\label{Lm:cr}
For any quadratic function $w\in P_2(K)$,
\begin{equation}\label{identity:cr}
\begin{array}{rl}
(I-\PiP)w=&P_{\rm P_1}^K(\nabla^2 w),\\
(I-\PiCR)w=&P_{\rm CR}^K(\nabla^2 w),\\
(I-\PiECR)w=&P_{\rm ECR}^K(\nabla^2 w).
\end{array}
\end{equation}
\end{Lm}

\subsection{Superconvergence results for the CR element}
Before designing a posteriori error estimates, we represent 
the post-processing mechanism which was first proposed in \cite{brandts1994superconvergence} and then analyzed in \cite{Hu2016Superconvergence} for the CR element. The shape function space of the Raviart-Thomas element \cite{Raviart1977A} reads as follows
$$
\text{RT}_K:=(P_0(K))^2+ \bold{x}P_0(K)\quad\text{ for any }K\in \cT_h,
$$
and the corresponding finite element space is
$$
\text{RT}(\cT_h):=\big \{\tau\in H(\text{div},\Om,\mathbb{R}^2): \tau|_K\in \text{RT}_K\text{ for any }K\in \cT_h\big \}.
$$
Given $\textbf{q}\in \rm RT(\cT_h)$, define function $K_h \textbf{q}|_K\in P_1(K)\times P_1(K)$ as in \cite{brandts1994superconvergence,Hu2016Superconvergence}.
\begin{Def}\label{Def:R}
1.For each interior edge $e\in\cE_h^i$, the elements $K_e^1$ and $K_e^2$ are the pair of elements sharing $e$. Then the value of $K_h \textbf{q}$ at the midpoint $\textbf{m}_e$ of $e$ is
$$
K_h \textbf{q}(\textbf{m}_e)={1\over 2}(\textbf{q}|_{K_e^1}(\textbf{m}_e)+\textbf{q}|_{K_e^2}(\textbf{m}_e)).
$$

2.For each boundary edge $e\in\cE_h^b$, let $K$ be the element having $e$ as an edge, and $K'$ be an element sharing an edge $e'\in\cE_h^i$ with $K$. Let $e''$ denote the edge of $K'$ that does not intersect with $e$, and $\textbf{m}$, $\textbf{m}'$ and $\textbf{m}''$ be the midpoints of the edges $e$, $e'$ and $e''$, respectively. Then the value of $K_h \textbf{q}$ at the point $\textbf{m}$ is
$$
K_h \textbf{q}(\textbf{m})=2K_h \textbf{q}(\textbf{m}') - K_h \textbf{q}(\textbf{m}'').
$$
\begin{center}
\begin{tikzpicture}[xscale=2,yscale=2]
\draw[-] (-0.5,0) -- (2,0);
\draw[-] (0,0) -- (0.5,1);
\draw[-] (0.5,1) -- (2,1);
\draw[-] (0.5,1) -- (1.5,0);
\draw[-] (2,1) -- (1.5,0);
\node[below, right] at (1,0.5) {\textbf{m}'};
\node[above] at (1.25,1) {\textbf{m}''};
\node[below] at (0.75,0) {\textbf{m}};
\node at (0.7,0.4) {K};
\node at (1.4,0.75) {K'};
\node at (1,0) {e};
\node at (1.4,0.2) {e'};
\node at (1.7,1) {e''};
\node at (0.3,-0.1) {$\partial \Om $};
\fill(1,0.5) circle(0.5pt);
\fill(1.25,1) circle(0.5pt);
\fill(0.75,0) circle(0.5pt);
\end{tikzpicture}
\end{center}
\end{Def}


The CR element solution for source problems admits a one order superconvergence on uniform triangulations \cite{hu2018optimal}. According to \cite{hu2019asymptotic}, the eigenfunction $ \CRu$ superconverges to the CR element solution for the corresponding source problem.
These two facts lead to the following superconvergence result
\begin{equation}\label{CRsuper}
\parallel \nabla u - K_h\nabla_h \CRu \parallel_{0,\Om}\leq C  h^2|\ln h|^{1/2}|u |_{\frac{7}{2},\Om},
\end{equation}
provided that $u\in H^{\frac{7}{2}}(\Om,\mathbb{R})\cap H^1_0(\Om,\mathbb{R})$. A similar same order superconvergence result of the CR element on a mildly structured mesh with a somehow higher regularity assumption was analyzed in \cite{li2018global} for source problems, which can be extended to eigenvalue problems. This superconvergence property \eqref{CRsuper} leads to the following lemma for second order derivatives of eigenfunctions.

\begin{Lm}\label{highorderestimate}
Let $(\lambda , u )$ be an eigenpair of \eqref{variance} with $u \in H^{\frac{7}{2}}(\Om,\mathbb{R})\cap H^1_0(\Om,\mathbb{R})$, and $(\CRlam,\CRu)$ be the corresponding approximate eigenpair of \eqref{discrete} in $\VCR$. It holds that
\begin{equation*}
\parallel \nabla^{2} u -\nabla_h K_h\nabla_h \CRu \parallel_{0,\Om}\leq C  h|\ln h|^{1/2}|u |_{\frac{7}{2},\Om},
\end{equation*}
provided that $u\in H^{\frac{7}{2}}(\Om,\mathbb{R})\cap H^1_0(\Om,\mathbb{R})$.
\end{Lm}
\begin{proof}
Let $ \Pi_{\rm P_2} $ be the second order Lagrange interpolation, namely, the interpolation $ \Pi_{\rm P_2} u$ is a piecewise quadratic function over $\cT_h$ and admits the same value as $u $ at the vertices of each element and the midpoint of each edge. It follows from the theory in \cite{ShiWangBook} that
\begin{equation}\label{polationerr}
\big| u -\Pi_{\rm P_2} u \big|_{i,\Om}\leq C  h^{3-i}|u |_{3,\Om},\ 0\leq i\leq 2.
\end{equation}
Due to the triangle inequality,
\begin{equation}\label{uRHu2014Lowerhtotal}
\resizebox{0.9\textwidth}{!}{ $
\parallel \nabla^2 u -\nabla_h K_h\nabla_h \CRu \parallel_{0,\Om}\leq \parallel \nabla^2 u -\nabla^2_h \Pi_{\rm P_2} u \parallel_{0,\Om} + \parallel \nabla^2_h \Pi_{\rm P_2} u -\nabla_h K_h\nabla_h \CRu \parallel_{0,\Om}.
$}
\end{equation}
By the inverse inequality,
\begin{equation}\label{inverseineq}
\parallel \nabla^{2}_h \Pi_{\rm P_2} u -\nabla_h K_h\nabla_h \CRu \parallel_{0,\Om}\leq C  h^{-1}\parallel \nabla_h \Pi_{\rm P_2} u -K_h\nabla_h \CRu  \parallel_{0,\Om}.
\end{equation}
A combination of \eqref{CRsuper}, \eqref{polationerr} and  \eqref{inverseineq} yields
\begin{equation}\label{Rhinterpolation}
\begin{split}
&\parallel\nabla^{2}_h\Pi_{\rm P_2} u -\nabla_h K_h\nabla_h \CRu \parallel_{0,\Om}\\
\leq &C  h^{-1}| \Pi_{\rm P_2} u - u |_{1,\Omega}+h^{-1}\parallel \nabla u -K_h\nabla_h \CRu  \parallel_{0,\Om}\ 
\leq  C  h|\ln h|^{1/2}|u |_{\frac{7}{2},\Om}.
\end{split}
\end{equation}
A substitution of \eqref{polationerr} and \eqref{Rhinterpolation} into \eqref{uRHu2014Lowerhtotal} concludes
\begin{equation*}
\parallel \nabla^{2} u -\nabla_h K_h\nabla_h \CRu \parallel_{0,\Om}\leq C  h|\ln h|^{1/2}|u |_{\frac{7}{2},\Om}
\end{equation*}
which completes the proof.
\qed
\end{proof}

\section{Asymptotically exact a posteriori error estimates}\label{sec:errorestimate}\label{sec:posteriori}
In this section, asymptotically exact a posteriori error estimates for eigenvalues are designed and analyzed for the CR element and the ECR element.

For eigenvalues of the Laplacian operator solved by the conforming linear element, asymptotically exact a posteriori error estimates in \cite{Zhang2006Enhancing}  are based on a simple identity
$$
\lambda_h -\lambda =|u -u_h|_{1,\Om}^2-\lambda  \parallel u -u_h \parallel_{0,\Om}^2,
$$
where $(\lambda_h, u_h)$ is an approximate eigenpair by this conforming element. The $L^2$ norm of the error in the approximation of eigenfunctions is of higher order compared to their energy norm.
By approximating the first term $|u -u_h|_{1,\Om}^2$ with high accuracy gradient recovery techniques \cite{huang2010superconvergent,Zhang2005A,Zienkiewicz1992The}, asymptotically exact a posteriori error estimates for eigenvalues are resulted following the above identity.

For nonconforming elements of second order elliptic eigenvalue problems, there exists  a similar identity to the one  in \cite{yang2000posteriori}. 
\begin{lemma}
Let $(\lambda, u)$ be an eigenpair of \eqref{variance} and $(\lambda_h, u_h)$ the corresponding approximate eigenpair of \eqref{discrete}. There holds
\begin{equation}\label{originalId}
\lambda -\lambda_h =|u - u_h |_{1, h}^2+2 (a_h(u ,u_h )-\lambda_h (u , u_h ))- \lambda_h \|u - u_h\|_{0, \Om}^2.
\end{equation}
\end{lemma}
\begin{proof}
By \eqref{variance} and \eqref{discrete},
$$
a_h(u-u_h,u-u_h)=a(u,u)+a_h(u_h,u_h)-2a_h(u,u_h)=\lambda + \lambda_h -2a_h(u,u_h).
$$
It follows that
\begin{equation}\label{idproof}
\lambda -\lambda_h =a_h(u-u_h,u-u_h)+2 a_h(u ,u_h )-2\lambda_h.
\end{equation}
Note that
$
 \|u - u_h\|_{0, \Om}^2 + 2(u,u_h)=2.
$
The identity \eqref{idproof} reads
$$
\lambda -\lambda_h =a_h(u-u_h,u-u_h)+2 (a_h(u ,u_h )-\lambda_h (u , u_h ))- \lambda_h \|u - u_h\|_{0, \Om}^2,
$$
which completes the proof.
\end{proof}

Compared to the identity for conforming elements, the identity \eqref{originalId} includes an extra term $a_h(u ,u_h )-\lambda_h (u , u_h )$. 
The nonconformity leads to this consistency error term
, which 
relates to values of the eigenfunctions. Since there is no high accuracy techniques in literature to recover eigenfunctions themselves, this extra term causes the difficulty to approximate discrete eigenvalues with high accuracy. 

\subsection{First type of asymptotically exact a posteriori error estimates}
For both the CR element and the ECR element, the canonical interpolation admits a commuting property \eqref{commuting}. By subtracting \eqref{commuting} from the extra term, the aforementioned consistency error can be expressed in terms of  the interpolation error. Taylor expansions in Lemma \ref{Lm:cr} imply that the dominant ingredient of the interpolation error is the second order derivatives of the eigenfunctions. This important property turns the function recovery problem into a gradient recovery problem. To be specific, for the CR element, thanks to the commuting property \eqref{commuting} of the canonical interpolation operator $\PiCR$,
\begin{equation}\label{commutId}
\lambda -\CRlam =|u - \CRu |_{1, h}^2-2\CRlam (u -\PiCR u ,\CRu )-\CRlam \|u -\CRu\|_{0,\Om}^2.
\end{equation}
The term $|u - \CRu |_{1, h}^2$ can be approximated with high accuracy by a direct application of gradient recovery techniques \cite{huang2010superconvergent,Zhang2005A,Zienkiewicz1992The} or superconvergent results in \cite{hu2018optimal,li2018global}. 
By the Taylor expansions in Lemma \ref{Lm:cr} and the superconvergent result in Lemma \ref{highorderestimate}, the interpolation error term $(u -\PiCR u ,\CRu )$ can be approximated with high accuracy by the use of these gradient recovery techniques. Then, asymptotically exact a posteriori error estimates for eigenvalues are designed from the identity \eqref{commutId}. This idea also works for eigenvalues by the ECR element.

Note that  within each element $K$, both $\nabla_h K_h\nabla_h \CRu$ and  $\nabla_h K_h\nabla_h \ECRu$ belong to $  P_0(K, \mathbb{R}^{2\times 2})$. Define the following a posteriori error estimates
\begin{equation}\label{FdefineCR}
\resizebox{0.91\textwidth}{!}{ $
F_{\rm CR, 1}^{\rm CR}=\parallel K_h\nabla_h \CRu  -\nabla_h \CRu\parallel_{0,\Om}^2-2\CRlam \sum_{K\in\cT_h} \int_K P_{\rm CR}^K(\nabla_h K_h\nabla_h \CRu ) \CRu\dx,
$}
\end{equation}
\begin{equation}\label{FdefineECR}
\resizebox{0.91\textwidth}{!}{ $
F_{\rm ECR, 1}^{\rm ECR}=\parallel K_h\nabla_h \ECRu  -\nabla_h \ECRu\parallel_{0,\Om}^2-2\ECRlam \sum_{K\in\cT_h} \int_K P_{\rm ECR}^K(\nabla_h K_h\nabla_h \ECRu )\ECRu\dx
$}
\end{equation}
with the polynomials $P_{\rm CR}^K$ and $P_{\rm ECR}^K$ defined in \eqref{Pdefine}.

\begin{Th}\label{Th:cr}
Let $(\lambda , u )$ be an eigenpair of \eqref{variance} with $u \in H^{\frac{7}{2}}(\Om,\mathbb{R})\cap H^1_0(\Om,\mathbb{R})$, and $(\CRlam,\CRu)$ be the corresponding approximate eigenpair of \eqref{discrete} in $\VCR$. The a posteriori error estimate $F_{\rm CR, 1}^{\rm CR}$ in \eqref{FdefineCR} satisfies
$$
\big |\lambda -\CRlam-F_{\rm CR, 1}^{\rm CR}\big |\leq C  h^{3}|\ln h|^{1/2}|u |_{\frac{7}{2},\Om}^2.
$$
\end{Th}
\begin{proof}
By the definition of $F_{\rm CR, 1}^{\rm CR}$ in \eqref{FdefineCR} and \eqref{commutId},
\begin{equation}\label{errorexpand}
\resizebox{0.91\textwidth}{!}{ $
\begin{split}
\lambda -\CRlam-F_{\rm CR, 1}^{\rm CR}=&|u -\CRu |_{1, h}^2 - \parallel K_h\nabla_h \CRu  -\nabla_h \CRu\parallel_{0,\Om}^2-\CRlam\|u -\CRu\|_{0,\Om}^2\\
& -2\CRlam\sum_{K\in\cT_h} \big (u -\PiCR u -P_{\rm CR}^K(\nabla^2 u ),\CRu\big )_{0,K}\\
&-2\CRlam\sum_{K\in\cT_h} \big (P_{\rm CR}^K(\nabla^2 u )-P_{\rm CR}^K(\nabla_h K_h\nabla_h \CRu  ),\CRu\big )_{0,K}.
\end{split}
$}
\end{equation}
Thanks to \eqref{CR:est} and \eqref{CRsuper},
\begin{equation}\label{1term}
\begin{split}
\big ||u -\CRu |_{1, h}^2 - \parallel K_h\nabla_h \CRu  -\nabla_h \CRu\parallel_{0,\Om}^2\big |\leq C  h^{3}|\ln h|^{1/2}|u |_{\frac{7}{2},\Om}^2.
\end{split}
\end{equation}
It follows from the Bramble-Hilbert lemma and Lemma \ref{Lm:cr} that
\begin{equation}\label{2term}
\big | \sum_{K\in\cT_h}\big (u -\PiCR u -P_{\rm CR}^K(\nabla^2 u ),\CRu\big )_{0,K}\big |\leq C  h^3|u |_{3,\Om}.
\end{equation}
Due to the triangle inequality,
\begin{equation}\label{3temp1}
\begin{split}
&\big |\sum_{K\in\cT_h}\big (P_{\rm CR}^K(\nabla^2 u )-P_{\rm CR}^K(\nabla_h K_h\nabla_h \CRu  ),\CRu\big )_{0,K}\big |
\\
\leq & \sum_{K\in\cT_h} \|P_{\rm CR}^K(\nabla^2 u )-P_{\rm CR}^K(\nabla_h K_h\nabla_h \CRu) \|_{0,K} \|\CRu\|_{0,K} .
\end{split}
\end{equation}
According to Lemma \ref{Lm:cr} and the Bramble-Hilbert lemma, 
$$
\|P_{\rm CR}^K(H_h ) \|_{0,K}\leq C h^2\|H_h\|_{0,K}.
$$
By the definition of $P_{\rm CR}^K$, 
\begin{equation}\label{3temp2}
\|P_{\rm CR}^K(\nabla^2 u )-P_{\rm CR}^K(\nabla_h K_h\nabla_h \CRu) \|_{0,K}\leq C h^2\|\nabla^2 u-K_h\nabla_h \CRu \|_{0,K}.
\end{equation}
A combination of \eqref{3temp1}, \eqref{3temp2} and Lemma \ref{highorderestimate} leads to
\begin{equation}\label{3term}
\big |\sum_{K\in\cT_h}\big (P_{\rm CR}^K(\nabla^2 u )-P_{\rm CR}^K(\nabla_h K_h\nabla_h \CRu  ),\CRu\big )_{0,K}\big |
\leq C   h^{3}|\ln h|^{1/2}|u |_{\frac{7}{2},\Om}^2.
\end{equation}
A substitution of \eqref{CR:est}, \eqref{1term}, \eqref{2term} and \eqref{3term} into \eqref{errorexpand} concludes
$$
\big |\lambda -\CRlam-F_{\rm CR, 1}^{\rm CR}\big |\leq C  h^{3}|\ln h|^{1/2}|u |_{\frac{7}{2},\Om}^2
$$
and completes the proof.
\qed
\end{proof}

Compared to the terms in \eqref{commutId}, the above analysis shows that the term 

\noindent $\parallel K_h\nabla_h \CRu  -\nabla_h \CRu\parallel_{0,\Om}^2$  in $F_{\rm CR, 1}^{\rm CR}$ approximates $|u - \CRu |_{1, h}^2$  with high accuracy, and $\sum_{K\in\cT_h} \int_K P_{\rm CR}^K(\nabla_h K_h\nabla_h \CRu ) \CRu\dx$ approximates  $ (u -\PiCR u ,\CRu )$ with high accuracy. Notice that other a posteriori error estimates can be constructed following \eqref{FdefineCR}  with different recovered gradients from $K_h\nabla_h \CRu $. The resulted a posteriori error estimates  are asymptotically exact as long as the recovered gradient admits a superconvergence result.

Similarly, the a posteriori error estimate $F_{\rm ECR, 1}^{\rm ECR}$ in \eqref{FdefineECR} is also asymptotically exact as presented in the following theorem.
\begin{Th}\label{Th:ecr}
Let $(\lambda , u )$ be an eigenpair of \eqref{variance} with $u \in H^{\frac{7}{2}}(\Om,\mathbb{R}) \cap  H^1_0(\Om,\mathbb{R})$, and $(\ECRlam,\ECRu)$ be the corresponding approximate eigenpair of \eqref{discrete} in $\VECR$. Then,
$$
\big |\lambda -\ECRlam-F_{\rm ECR, 1}^{\rm ECR}\big |\leq C  h^{3}|\ln h|^{1/2}|u |_{\frac{7}{2},\Om}^2.
$$
\end{Th}
\begin{remark}\label{remark:morley}
Suppose that $(\lambda , u )$ is an eigenpair of the biharmonic operator with $u \in H^{\frac{9}{2}}(\Om,\mathbb{R})\cap H^2_0(\Om,\mathbb{R})$, and $(\Mlam, \Mu)$ is the corresponding approximate eigenpair by the Morley element on an uniform triangulation $\cT_h$. Thanks to the superconvergence of the Hellan--Herrmann--Johnson element and its equivalence to the Morley element, the recovered Hessian $K_h\nabla_h^2 \Mu$ superconverges to $ \nabla^2 u$. Since the canonical interpolation operator of the Morley element also admits the commuting property, a similar procedure produces asymptotically exact a posteriori error estimates for eigenvalues by the Morley element.
\end{remark}

\subsection{Second type of asymptotically exact a posteriori error estimates}
The second type of asymptotically exact a posteriori error estimates works for any nonconforming elements as long as the corresponding discrete space contains a conforming subspace and there exists some high accuracy gradient recovery technique for the elements.

The canonical interpolation of a conforming element is employed here to approximate the consistency error term. 
Take the CR element for example,
\begin{equation}\label{PostId2}
\begin{split}
\lambda -\CRlam =&|u -\CRu |_{1, h}^2+2a( u-\PiP u, \CRu)-2\CRlam (u -\PiP u ,\CRu )\\
&-\CRlam \|u -\CRu\|_{0,\Om}^2,
\end{split}
\end{equation}
and the canonical interpolation $\PiP u$ of the conforming linear element admits the same value of $u$ on each vertex.
The main idea here is to rewrite the tricky term $a( u-\PiP u, \CRu)$
by the Green identity as follows.
\begin{equation}\label{Post2average}
a( u-\PiP u, \CRu)=
\sum_{e\in \cE_h^i} \int_{e}(u-\PiP u)[{\partial \CRu\over \partial n}]\ds.
\end{equation}
On each interior edge, the interpolation error is approximated by the average of recovered interpolation errors on two adjacent elements. 

Define the following a posteriori error estimates
\begin{equation}\label{FdefineCR2}
\resizebox{0.9\textwidth}{!}{ $
\begin{array}{rl}
F_{\rm CR, 2}^{\rm CR}=&\parallel K_h\nabla_h \CRu  -\nabla_h \CRu\parallel_{0,\Om}^2
+2\sum\limits_{e\in \cE_h^i} \int_{e}\big \{P_{\rm P_1}^K(\nabla_h K_h\nabla_h \CRu )\big \}\big [\frac{\partial \CRu}{\partial n}\big ]\ds\\
&-2\CRlam \sum\limits_{K\in\cT_h} \int_K P_{\rm P_1}^K(\nabla_h K_h\nabla_h \CRu ) \CRu\dx,
\end{array}
$}
\end{equation}
\begin{equation}\label{FdefineECR2}
\resizebox{0.9\textwidth}{!}{ $
\begin{split}
F_{\rm ECR, 2}^{\rm ECR}=&\parallel K_h\nabla_h \ECRu  -\nabla_h \ECRu\parallel_{0,\Om}^2
+2\sum\limits_{e\in \cE_h^i} \int_{e}\big \{P_{\rm P_1}^K( \nabla_h K_h\nabla_h \ECRu)\big \}\big [\frac{\partial \ECRu}{\partial n}\big ]\ds\\
&-2\sum\limits_{K\in\cT_h} \int_K P_{\rm P_1}^K(\nabla_h K_h\nabla_h \ECRu )( \Delta_h \ECRu + \ECRlam \ECRu)\dx,
\end{split}
$}
\end{equation}
with the polynomial $P_{\rm P_1}^K( \cdot )$ defined in \eqref{Pdefine}.

\begin{Th}\label{Th:cr2}
Let $(\lambda , u )$ be an eigenpairs of \eqref{variance} with $u \in H^{\frac{7}{2}}(\Om,\mathbb{R})\cap H^1_0(\Om,\mathbb{R})$, and $(\CRlam,\CRu)$ be the corresponding approximate eigenpairs of \eqref{discrete} in $\VCR$. The a posteriori error estimate $F_{\rm CR, 2}^{\rm CR}$ in \eqref{FdefineCR} satisfies
$$
\big |\lambda -\CRlam-F_{\rm CR, 2}^{\rm CR}\big |\leq C  h^{3}|\ln h|^{1/2}|u |_{\frac{7}{2},\Om}^2.
$$
\end{Th}
\begin{proof}
According to the analysis of Theorem \ref{Th:cr}, the definition of $F_{\rm CR, 2}^{\rm CR}$, \eqref{PostId2} and \eqref{Post2average}, it only remains to prove that
\begin{equation*}
|\sum_{e\in \cE_h^i} \int_{e}(u-\PiP u - \big \{P_{\rm P_1}^K(\nabla_h K_h\nabla_h \CRu )\big \})\big [\frac{\partial \CRu}{\partial n}\big ]\ds
\leq C  h^3|\ln h|^{1/2}|u |_{\frac{7}{2},\Om}^2.
\end{equation*}
The Cauchy-Schwarz inequality implies that
\begin{equation}\label{CRterm1}
\begin{split}
&\int_{e}(u-\PiP u - \big \{P_{\rm P_1}^K(\nabla_h K_h\nabla_h \CRu )\big \})\big [\frac{\partial \CRu}{\partial n}\big ]\ds \\
\leq C  &\|u-\PiP u - \big \{P_{\rm P_1}^K(\nabla_h K_h\nabla_h \CRu )\big \}\|_{0, e}\| \big [\frac{\partial (\CRu - u)}{\partial n}\big ]\|_{0, e}.
\end{split}
\end{equation}
Due to the trace theorem and \eqref{CR:est},
\begin{equation}\label{CRterm2}
\| \big [\frac{\partial (\CRu - u)}{\partial n}\big ]\|_{0, e}\leq C  h^{1/2} |u |_{2,\omega_e}.
\end{equation}
According to the Bramble-Hilbert lemma, the trace theorem and Lemma \ref{Lm:cr},
\begin{equation}\label{CRterm3}
\|u-\PiP u - P_{\rm P_1}^K(\nabla^2 u )\|_{0, e}\leq C  h^{5/2}|u|_{3, \omega_e}.
\end{equation}
A combination of Lemma \ref{highorderestimate} and the trace theorem  gives
\begin{equation}\label{CRterm4}
\|P_{\rm P_1}^K(\nabla^2 u ) - P_{\rm P_1}^K(\nabla_h K_h \nabla_h \CRu )\|_{0, e}\leq C  h^{5/2}|\ln h|^{1/2}|u |_{\frac{7}{2},\omega_e}.
\end{equation}
A substitution of \eqref{CRterm2}, \eqref{CRterm3} and \eqref{CRterm4} into \eqref{CRterm1} concludes
\begin{equation*}
|\sum_{e\in \cE_h^i} \int_{e}(u-\PiP u - \big \{P_{\rm P_1}^K(\nabla_h K_h\nabla_h \CRu )\big \})\big [\frac{\partial \CRu}{\partial n}\big ]\ds
\leq C  h^3|\ln h|^{1/2}|u |_{\frac{7}{2},\Om}^2
\end{equation*}
and completes the proof.
\end{proof}

Compared to the terms in \eqref{PostId2}, the above analysis shows that the term 

\noindent $\parallel K_h\nabla_h \CRu  -\nabla_h \CRu\parallel_{0,\Om}^2$, 
$\sum\limits_{e\in \cE_h^i} \int_{e}\big \{P_{\rm P_1}^K(\nabla_h K_h\nabla_h \CRu )\big \}\big [\frac{\partial \CRu}{\partial n}\big ]\ds$
and 

\noindent $\sum\limits_{K\in\cT_h} \int_K P_{\rm P_1}^K(\nabla_h K_h\nabla_h \CRu ) \CRu\dx $   in the a posteriori error estimate 
$F_{\rm CR, 2}^{\rm CR}$ approximate  $|u - \CRu |_{1, h}^2$, 
$a( u-\PiP u, \CRu)$ and $(u -\PiP u ,\CRu )$ with high accuracy, respectively.

The a posteriori error estimates $F_{\rm ECR, 2}^{\rm ECR}$ in \eqref{FdefineECR2} are designed following the procedure for the CR element. Different from \eqref{Post2average} for the CR element, 
\begin{equation}\label{ecr:deco}
a(u-\PiP u, \ECRu)=\sum_{e\in \cE_h} \int_{e}(u-\PiP u)[{\partial \ECRu\over \partial n}]\ds - (u-\PiP u, \Delta_h \ECRu),
\end{equation}
including an extra term $\sum_{K\in\cT_h}\int_K (u-\PiP u)\Delta_h \ECRu\dx$. Thus, the error estimate $F_{\rm ECR, 2}^{\rm ECR}$ for the ECR element contains a corresponding term
$$
-2\sum\limits_{K\in\cT_h} \int_K P_{\rm P_1}^K(\nabla_h K_h\nabla_h \ECRu ) \Delta_h \ECRu \dx
$$
to approximate the aforementioned extra term with high accuracy. A similar analysis to the one in Theorem \ref{Th:cr2} proves that the a posteriori error estimate $F_{\rm ECR, 2}^{\rm ECR}$ in \eqref{FdefineECR2} is also asymptotically exact. 
The analysis in Theorem  \ref{Th:cr2} for the second type of a posteriori error estimates is of optimal order as verified in Figure \ref{fig:squareCRECRP1} and Figure \ref{fig:squareNeumannCRECRP1}.
Note that other asymptotically exact a posteriori error estimates can be constructed following \eqref{FdefineCR2} and \eqref{FdefineECR2} but with different high accuracy recovered gradients. 

%

\section{Postprocessing algorithm}\label{sec:algo}
This section proposes two methods to improve accuracy of approximate eigenvalues by employing asymptotically exact a posteriori error estimates.

For the ease of presentation, we list the notations of different approximate eigenvalues and a posteriori error estimates here. Denote the approximate eigenpairs by the CR element, the ECR element and the conforming linear element on $\cT_h$ by $(\CRlam , \CRu )$, $(\ECRlam , \ECRu )$ and  $(\Ponelam, \Poneu)$, respectively. 
Apply the average-projection in \cite{Hu2015Constructing} to the approximate eigenfunctions  $\CRu $, and denote the resulted eigenfunctions by $\tilde{u}_{\rm P_1^{\ast}}$. Define
\begin{equation}\label{projectionP1}
\PoneAstu:= \tilde{u}_{\rm P_1^{\ast}}/\parallel \tilde{u}_{\rm P_1^{\ast}}\parallel_{0,\Om}\ \text{ and }\ \PoneAstlam:=a_h(\PoneAstu,\PoneAstu).
\end{equation}
Denote the PPR postprocessing technique in \cite{Zhang2005A} by the operator $\overline{K}_h$ and 
$$
F_{\rm P_1}^{\rm P_1}:=\parallel \overline{K}_h\nabla \Poneu  -\nabla \Poneu\parallel_{0,\Om}^2.
$$
According to \cite{Zhang2006Enhancing}, $F_{\rm P_1}^{\rm P_1}$ is an asymptotically exact a posteriori error estimate for eigenvalues $\Ponelam$.  Replace the recovered gradient $\overline{K}_h\nabla \Poneu$ in the above definition by $K_h\nabla_h \CRu$ and $\overline{K}_h\nabla \PoneAstu$, and denote the resulted a posteriori error estimates by $F_{\rm P_1}^{\rm CR}$ and $F_{\rm P_1}^{\rm P_1^\ast}$, respectively.   Lemma \ref{highorderestimate} reveals that the error estimate $F_{\rm P_1}^{\rm CR}$ is asymptotically exact.

For eigenvalues by the CR element, replace the recovered gradient $K_h\nabla_h \CRu$ in \eqref{FdefineCR} by $\overline{K}_h\nabla \Poneu$ and $\overline{K}_h\nabla \PoneAstu$, and denote the resulted first type of a posteriori error estimates by $F_{\rm CR, 1}^{\rm P_1}$ and $F_{\rm CR, 1}^{\rm P_1^\ast}$, respectively. Similarly, replace the recovered gradient $K_h\nabla_h \CRu$ in \eqref{FdefineCR2} by $\overline{K}_h\nabla \Poneu$ and $\overline{K}_h\nabla \PoneAstu$, and denote the resulted second type of a posteriori error estimates by $F_{\rm CR, 2}^{\rm P_1}$ and $F_{\rm CR, 2}^{\rm P_1^\ast}$, respectively. According to \cite{Zhang2005A}, the recovered gradient $\overline{K}_h\nabla \Poneu$ admits a similar superconvergence result as $K_h\nabla_h \CRu$ in \eqref{CRsuper}. Thus, both error estimates $F_{\rm CR, 1}^{\rm P_1}$ and $F_{\rm CR, 2}^{\rm P_1}$ are asymptotically exact too.

\subsection{Recovered eigenvalues}

The first approach is to correct the discrete eigenvalues by the corresponding asymptotically exact a posteriori error estimates. A direct application of Theorem \ref{Th:cr} and Theorem \ref{Th:cr2} proves the higher accuracy of the resulted recovered eigenvalues than that of the original ones.

\begin{table}[!ht]
\footnotesize
\label{tab:definition}
\begin{tabular}{|c|c|c|c|c|}
\hline
 & Recovered gradients &  \multicolumn{2}{c|}{Recovered eigenvalues}  &  order\\ \hline
\multirow{3}{*}{$\CRlam$} &  $K_h\nabla_h \CRu$& $\lambda_{\rm CR, 1}^{\rm R,\ CR}:=\CRlam + F_{\rm CR, 1}^{\rm CR}$ &  $\lambda_{\rm CR, 2}^{\rm R,\ CR}:=\CRlam + F_{\rm CR, 2}^{\rm CR}$  &  3\\ \cline{2-5} 
                  &  $\overline{K}_h\nabla \Poneu$ & $\lambda_{\rm CR, 1}^{\rm R,\ P_1}:=\CRlam + F_{\rm CR, 1}^{\rm P_1}$ &  $\lambda_{\rm CR, 2}^{\rm R,\ P_1}:=\CRlam + F_{\rm CR, 2}^{\rm P_1}$ & 3 \\ \cline{2-5} 
                  &  $\overline{K}_h\nabla \PoneAstu$&  $\lambda_{\rm CR, 1}^{\rm R,\ P_1^\ast}:=\CRlam + F_{\rm CR, 1}^{\rm P_1^\ast}$ &  $\lambda_{\rm CR, 2}^{\rm R,\ P_1^\ast}:=\CRlam + F_{\rm CR, 2}^{\rm P_1^\ast}$& -  \\ \hline
\multirow{3}{*}{$\Ponelam$} & $K_h\nabla_h \CRu$ & \multicolumn{2}{c|}{$\lambda_{\rm P_1}^{\rm R,\ CR}:=\Ponelam - F_{\rm P_1}^{\rm CR}$ }&  3\\ \cline{2-5} 
                  &  $\overline{K}_h\nabla \Poneu$  &\multicolumn{2}{c|}{$\lambda_{\rm P_1}^{\rm R,\ P_1}:=\Ponelam - F_{\rm P_1}^{\rm P_1}$} &  3\\ \cline{2-5} 
                  &  $\overline{K}_h\nabla \PoneAstu$&\multicolumn{2}{c|}{$\lambda_{\rm P_1}^{\rm R,\ P_1^\ast}:=\Ponelam - F_{\rm P_1}^{\rm P_1^\ast}$} &  -\\ \hline
\multirow{3}{*}{$\PoneAstlam$} &  $K_h\nabla_h \CRu$&\multicolumn{2}{c|}{$\lambda_{\rm P_1^\ast}^{\rm R,\ CR}:=\PoneAstlam - F_{\rm P_1^\ast}^{\rm CR}$}&  3\\ \cline{2-5} 
                  &  $\overline{K}_h\nabla \Poneu$ & \multicolumn{2}{c|}{$\lambda_{\rm P_1^\ast}^{\rm R,\ P_1}:=\PoneAstlam - F_{\rm P_1^\ast}^{\rm P_1}$} &  3\\ \cline{2-5} 
                  &   $\overline{K}_h\nabla \PoneAstu$& \multicolumn{2}{c|}{$\lambda_{\rm P_1^\ast}^{\rm R,\ P_1^\ast}:=\PoneAstlam - F_{\rm P_1^\ast}^{\rm P_1^\ast}$} &  - \\ \hline
\end{tabular}
\caption{\footnotesize Definitions of recovered eigenvalues and the corresponding convergence rates.}
\end{table}
Table \ref{tab:definition} lists the definitions of most recovered eigenvalues mentioned in this paper and their theoretical convergence rates. These theoretical convergence rates are direct results from Theorem \ref{Th:cr} and Theorem  \ref{Th:cr2}, 
and the superconvergence results of the recovered gradients $K_h\nabla_h \CRu$ and $\overline{K}_h\nabla \Poneu$.
Although there is no superconvergence result for the recovered gradient $\overline{K}_h\nabla \PoneAstu $, numerical examples still indicate a higher accuracy for the resulted recovered eigenvalues than the original ones.

\subsection{combined eigenvalues}

Another way to achieve high accuracy is to take a weighted-average of discrete eigenvalues with the weights computed by the corresponding asymptotically exact a posteriori error estimates. 
Usually the weighted-average of a lower bound, such as $\CRlam $ and $\ECRlam$, and an upper bound, like $\Ponelam$ and $\PoneAstlam$ can achieve better accuracy than those of two lower bounds or two upper bounds. 

\begin{table}[!ht]
\centering
\begin{tabular}{|c|c||c|c|c|c|c|}
\hline
                                        &  & \multicolumn{2}{c|}{$\Ponelam$} & \multicolumn{3}{c|}{$\PoneAstlam$} \\ \cline{3-7} 
                                        &  &   $F_{\rm P_1}^{\rm CR}$   & $F_{\rm P_1}^{\rm P_1}$ &    $F_{\rm P_1^\ast}^{\rm CR}$   & $F_{\rm P_1^\ast}^{\rm P_1}$  &     $F_{\rm P_1^\ast}^{\rm P_1^\ast}$  \\ \hline\hline
\multicolumn{1}{|c|}{\multirow{4}{*}{$\CRlam$}} & $F_{\rm CR, 1}^{\rm CR}$ &      $\lambda^{\rm C,\ P_1}_{\rm CR, CR, 1}$     & $ \lambda^{\rm C,\ P_1}_{\rm CR, P_1,1}$  &$\lambda^{\rm C,\ P_1^\ast}_{\rm CR, CR, 1}$ &   $\lambda^{\rm C,\ P_1^\ast}_{\rm CR, P_1, 1}$    &   $\lambda^{\rm C,\ P_1^\ast}_{\rm CR, P_1^\ast, 1}$    \\ \cline{2-7} 
\multicolumn{1}{|c|}{}                  & $F_{\rm CR, 2}^{\rm CR}$ &        $\lambda^{\rm C,\ P_1}_{\rm CR, CR, 2}$     & $ \lambda^{\rm C,\ P_1}_{\rm CR, P_1, 2}$  &$\lambda^{\rm C,\ P_1^\ast}_{\rm CR, CR, 2}$ &   $\lambda^{\rm C,\ P_1^\ast}_{\rm CR, P_1, 2}$    &   $\lambda^{\rm C,\ P_1^\ast}_{\rm CR, P_1^\ast, 2}$    \\ \cline{2-7} 
\multicolumn{1}{|c|}{}                  &  $F_{\rm CR, 1}^{\rm P_1}$&   $\lambda^{\rm C,\ P_1}_{\rm P_1, CR, 1}$     & $ \lambda^{\rm C,\ P_1}_{\rm P_1, P_1,1}$  &$\lambda^{\rm C,\ P_1^\ast}_{\rm P_1, CR, 1}$ &   $\lambda^{\rm C,\ P_1^\ast}_{\rm P_1, P_1, 1}$    &   $\lambda^{\rm C,\ P_1^\ast}_{\rm P_1, P_1^\ast, 1}$     \\\cline{2-7} 
\multicolumn{1}{|c|}{}                  &  $F_{\rm CR, 2}^{\rm P_1}$&     $\lambda^{\rm C,\ P_1}_{\rm P_1, CR, 2}$     & $ \lambda^{\rm C,\ P_1}_{\rm P_1, P_1, 2}$  &$\lambda^{\rm C,\ P_1^\ast}_{\rm P_1, CR, 2}$ &   $\lambda^{\rm C,\ P_1^\ast}_{\rm P_1, P_1, 2}$    &   $\lambda^{\rm C,\ P_1^\ast}_{\rm P_1, P_1^\ast, 2}$      \\ \hline
\end{tabular}
\caption{\footnotesize Notations for varied combined eigenvalues.}
\label{tab:combined}
\end{table}

In this paper, lower bounds for eigenvalues are fixed as $\CRlam$, and upper bounds are $\Ponelam$ or $\PoneAstlam$ as shown in Table \ref{tab:combined}. As listed in the first two columns of the table, four different asymptotically exact a posteriori error estimates $F_{\rm CR, 1}^{\rm CR}$, $F_{\rm CR, 2}^{\rm CR}$, $F_{\rm CR, 1}^{\rm P_1}$ and $F_{\rm CR, 2}^{\rm P_1}$ of the lower bound $\CRlam$ are considered here to compute the weights. The first two rows list the two a posteriori error estimates $F_{\rm P_1}^{\rm CR}$ and $F_{\rm P_1}^{\rm P_1}$ for the upper bound  $\Ponelam$, and the three a posteriori error estimates $F_{\rm P_1^\ast}^{\rm CR}$, $F_{\rm P_1^\ast}^{\rm P_1}$ and $F_{\rm P_1^\ast}^{\rm P_1^\ast}$ for the upper bound $\PoneAstlam$. The notations for the resulted combined eigenvalues are listed in the right bottom block of Table \ref{tab:combined}. Take the notation $\lambda^{\rm C,\ P_1^{\ast}}_{\rm P_1,CR, 1}$ for example: define
\begin{equation}\label{combineEig}
\lambda^{\rm C,\ P_1^{\ast}}_{\rm P_1,CR, 1}:=\frac{F_{\rm CR, 1}^{\rm P_1}}{F_{\rm CR, 1}^{\rm P_1} + F_{\rm P_1^{\ast}}^{\rm CR}}\PoneAstlam + \frac{F_{\rm P_1^{\ast}}^{\rm CR}}{F_{\rm CR, 1}^{\rm P_1} + F_{\rm P_1^{\ast}}^{\rm CR}}\CRlam .
\end{equation}
The combined eigenvalues $\lambda^{\rm C,\ P_1^{\ast}}_{\rm P_1,CR, 1}$ are the weighted-average of the upper bound  $\PoneAstlam$ and the lower bounds $\CRlam$. Note that 
\begin{align}
|\lambda -\PoneAstlam + F_{\rm P_1^\ast}^{\rm CR} |&\leq Ch^3,\label{c2}\\
|\lambda -\CRlam - F_{\rm CR, 1}^{\rm P_1} |&\leq Ch^3.\label{c1}
\end{align}
Multiplying \eqref{c2} by $\frac{F_{\rm CR, 1}^{\rm P_1}}{F_{\rm CR, 1}^{\rm P_1} + F_{\rm P_1^{\ast}}^{\rm CR}}$ and \eqref{c1} by $\frac{F_{\rm P_1^{\ast}}^{\rm CR}}{F_{\rm CR, 1}^{\rm P_1} + F_{\rm P_1^{\ast}}^{\rm CR}}$ and adding the results together,
$$
|\lambda - \lambda^{\rm C,\ P_1^{\ast}}_{\rm P_1,CR, 1}|\leq Ch^3.
$$
This indicates a third order accuracy of the combined eigenvalues $\lambda^{\rm C,\ P_1^{\ast}}_{\rm P_1,CR, 1}$ in \eqref{combineEig}. The other combined eigenvalues are defined in a similar way to the one in \eqref{combineEig}. 




The difference between the combined eigenvalues we propose and those in \cite{Hu2012A} lies in the design of the approximate weights. In  \cite{Hu2012A}, two elements, which produce upper bounds and lower bounds of eigenvalues, respectively, are employed to solve eigenvalue problems on two successive meshes. The weights there are computed by the resulted four approximate eigenvalues, while the weights in this paper are computed by the corresponding a posteriori error estimates.

Most of the combined eigenvalues in Table \ref{tab:combined} require to solve two eigenvalue problems and to compute the corresponding a posteriori error estimates, or even one more eigenvalue problem for the recovered gradient. The only exception is the combined eigenvalues $\lambda^{\rm C,\ P_1^\ast}_{\rm CR, CR, 1}$, $\lambda^{\rm C,\ P_1^\ast}_{\rm CR, CR, 2}$, $\lambda^{\rm C,\ P_1^\ast}_{\rm CR, P_1^\ast, 1}$ and $\lambda^{\rm C,\ P_1^\ast}_{\rm CR, P_1^\ast, 2}$, since the eigenpair $(\PoneAstlam, \PoneAstu)$ is computed by the projection in  \cite{Hu2015Constructing} of the nonconforming eigenfunction $\CRu$.  A third order  convergence rate of the resulted combined eigenvalues $\lambda^{\rm C,\ P_1^\ast}_{\rm CR, CR, 1}$ and $\lambda^{\rm C,\ P_1^\ast}_{\rm CR, CR, 2}$  is guaranteed by the high accuracy recovered gradient $K_h\nabla_h \CRu$.

\section{Numerical examples}\label{sec:numerical}
This section presents six numerical tests for eigenvalues of the Laplacian operator. The first three examples deal with smooth eigenfunctions and the other three deal with singular eigenfunctions.

\subsection{Adaptive algorithm}
There exist some a posteriori error estimates for eigenvalues by nonconforming elements in literature \cite{duran2003posteriori,larson2000posteriori,li2009posteriori,verfurth1994posteriori}.
The a posteriori error estimates in Section \ref{sec:posteriori} can also be employed to improve the accuracy of eigenvalue problems by adaptive algorithms. Starting from an initial  grid $\cT_1$, the adaptive mesh refinement process operates in the following (widely used) way (see \cite{verfurth1996review}):
\begin{enumerate}
\item Set k=0.
\item Compute the eigenvalues from \eqref{discrete} on $\cT_k$ and denote them by $\CRlam^{\rm A}$. 
\item Compute the the second type of asymptotically exact a posteriori error estimates for eigenvalues. Take the CR element as example, compute
{\small
	\begin{equation}\label{eta2}
        \begin{array}{rl}
        F_{\rm CR, 2}^{\rm CR}=&\parallel G_h\nabla_h \CRu  -\nabla_h \CRu\parallel_{0,\Om}^2
+2\sum\limits_{e\in \cE_h^i} \int_{e}\big \{P_{\rm P_1}^K(\nabla_h G_h\nabla_h \CRu )\big \}\big [\frac{\partial \CRu}{\partial n}\big ]\ds\\
&-2\CRlam \sum\limits_{K\in\cT_h} \int_K P_{\rm P_1}^K(\nabla_h G_h\nabla_h \CRu ) \CRu\dx,
        \end{array}
        \end{equation}}
        with the postprocessing operator $G_h$ defined in \cite{guo2015gradient}.
\item Compute the local error indicators. In this section, let the local error indicators be the one in \cite{li2009posteriori} as below
	\begin{equation}\label{eta}
        \eta_K^2=\CRlam^2 h_K^2\parallel \CRu \parallel_{0, K}^2
        +\sum_{e\subset \partial K}h_e \|\big[\frac{\partial \CRu}{\partial n}\big ]\|_{0, e}^2+h_e \|\big[\frac{\partial \CRu}{\partial t}\big ]\|_{0,e}^2
        \end{equation}
         and $\eta^2=\sum_{K\in\cT} \eta_K^2$.
\item If $\eta$ is sufficiently small then stop. Otherwise, 
refine those elements $K\in\cT_k$ with $\eta_K>\theta \max_{K\in \cT_k} \eta_K$. 
\item Set $k=k+1$ and go to (2).
\end{enumerate}
Here $0 < \theta < 1$ is a fixed threshold. In the numerical experiments, we always
set $\theta = 0.3$. Since the superconvergence result \eqref{CRsuper} for the CR element requires triangulations to be uniform, we use the PPR postprocessing operator $G_h$ in \cite{guo2015gradient} for all adaptive triangulations in this section.

\subsection{Example 1}
In this example, the model problem \eqref{variance} on the unit square $\Om=(0,1)^2$ is considered. The exact eigenvalues are
$$\lambda =(m^2+n^2)\pi^2 ,\ m,\ n\ \text{are positive integers},$$
and the corresponding eigenfunctions are $u =2\sin (m\pi x_1)\sin (n\pi x_2)$. The domain is partitioned by uniform triangles. The level one triangulation $\cT_1$ consists of two right triangles, obtained by cutting the unit square with a north-east line. Each triangulation $\cT_i$ is refined into a half-sized triangulation uniformly, to get a higher level triangulation $\cT_{i+1}$.

\vskip 0.5 cm
\noindent \textbf{Recovered eigenvalues.} Figure \ref{fig:squareCRECRP1} plots the errors of the first approximate eigenvalues by the CR element, the ECR element, the conforming linear element and their corresponding recovered eigenvalues.
\begin{figure}[!ht]
\setlength{\abovecaptionskip}{0pt}
\setlength{\belowcaptionskip}{0pt}
\centering
\includegraphics[width=8cm,height=6cm]{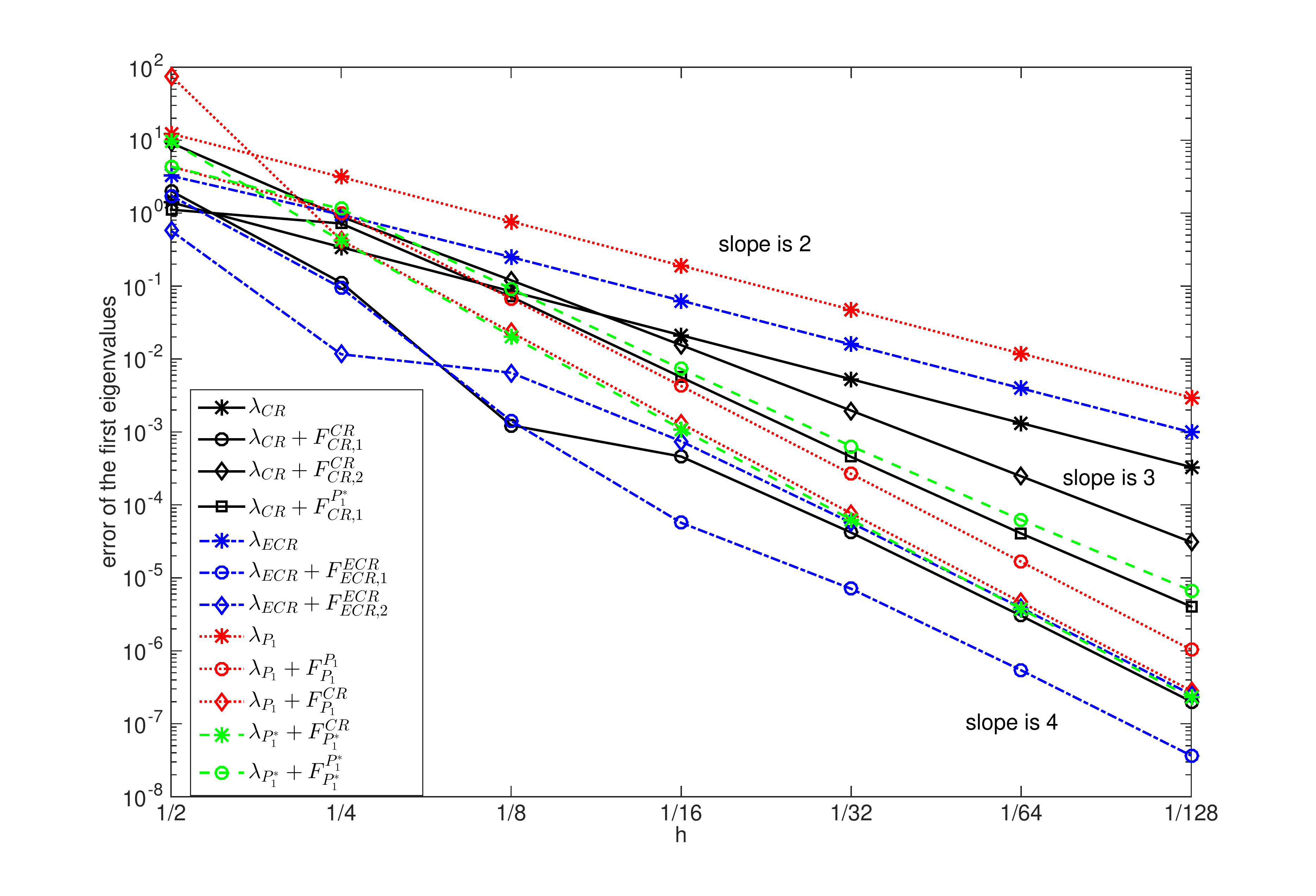}
\caption{\footnotesize{The errors of the recovered eigenvalues for Example 1.}}
\label{fig:squareCRECRP1}
\end{figure}
It shows that the approximate eigenvalues $\CRlam $, $\ECRlam $ and $\Ponelam$ converge at a rate 2, the recovered eigenvalues $\lambda_{\rm CR,\ 2}^{\rm R,\ CR}$ converge at a rate 3, and the recovered eigenvalues $\lambda_{\rm CR,\ 1}^{\rm R,\ CR}$, $\lambda_{\rm ECR,\ 1}^{\rm R,\ ECR}$, $\lambda_{\rm ECR,\ 2}^{\rm R,\ ECR}$,  and $\lambda_{\rm P_1}^{\rm R,\ P_1}$ converge at a higher rate 4. Note that although the theoretical convergence rate of the recovered eigenvalues $\lambda_{\rm CR,\ 1}^{\rm R,\ CR}$, $\lambda_{\rm ECR,\ 1}^{\rm R,\ ECR}$ and $\lambda_{\rm ECR,\ 2}^{\rm R,\ ECR}$ is only 3, numerical tests indicate an even higher convergence rate 4. The errors of the recovered eigenvalues $\lambda_{\rm CR,\ 1}^{\rm R,\ CR}$, $\lambda_{\rm CR,\ 2}^{\rm R,\ CR}$, $\lambda_{\rm ECR,\ 1}^{\rm R,\ ECR}$, $\lambda_{\rm ECR,\ 2}^{\rm R,\ ECR}$ and $\lambda_{\rm P_1}^{\rm R,\ P_1}$ on $\cT_8$  are $2.02\times 10^{-7}$, $3.09\times 10^{-5}$, $3.66\times 10^{-8}$, $2.52\times 10^{-7}$, $1.04\times 10^{-6}$, respectively. They are significant improvements on the errors of the approximate eigenvalues $\CRlam $, $\ECRlam $ and $\Ponelam$, which are $3.30\times 10^{-4}$,	 $9.91\times 10^{-4}$ and $2.97\times 10^{-3}$, respectively. This reveals that the recovered eigenvalues are quite remarkable improvements on the finite element solutions. It shows that the most accurate approximation is $\lambda_{\rm ECR,\ 1}^{\rm R,\ ECR}$, followed by  $\lambda_{\rm CR,\ 1}^{\rm R,\ CR}$, $\lambda_{\rm P_1^{\ast}}^{\rm R,\ CR}$, $\lambda_{\rm P_1}^{\rm R,\ CR}$ and  $\lambda_{\rm P_1}^{\rm R,\ P_1}$.  
Compared to the approximation $\lambda_{\rm P_1}^{\rm R,\ CR}$, the recovered eigenvalue  $\lambda_{\rm P_1^{\ast}}^{\rm R,\ CR}$ achieves a pretty much the same accuracy, but require to solve only one discrete eigenvalue problem.

\begin{table}[!ht]
  \centering
    \begin{tabular}{c|ccccc}
    \hline
    &$\lambda^{\rm C,\ P_1}_{\rm CR, CR,\ 1}$ & $\lambda^{\rm C,\ P_1}_{\rm CR, P_1,\ 1}$ & $\lambda^{\rm C,\ P_1^{\ast}}_{\rm CR, CR,\ 1}$ &$\lambda^{\rm C,\ P_1^{\ast}}_{\rm CR, P_1,\ 1}$ &$\lambda^{\rm C,\ P_1^{\ast}}_{\rm CR, P_1^{\ast},\ 1}$\\\hline
    error & 1.53E-07 & 7.79E-08 & 1.59E-07 &1.35E-08& 4.86E-07\\\hline
    &$\lambda^{\rm C,\ P_1}_{\rm P_1, CR,\ 1}$ & $\lambda^{\rm C,\ P_1}_{\rm P_1, P_1,\ 1}$ & $\lambda^{\rm C,\ P_1^{\ast}}_{\rm P_1, CR,\ 1}$ &$\lambda^{\rm C,\ P_1^{\ast}}_{\rm P_1, P_1,\ 1}$ &$\lambda^{\rm C,\ P_1^{\ast}}_{\rm P_1, P_1^{\ast},\ 1}$\\\hline
    error & 1.33E-06 & 1.25E-06 & 1.34E-06 & 1.19E-06& 6.90E-07\\\hline
    \end{tabular}
    \caption{\footnotesize The errors of different combined eigenvalues on the mesh $\cT_8$ for Example 1.}
    \label{tab:combination}%
\end{table}

\noindent \textbf{Combined eigenvalues.} The errors of some combined eigenvalues on $\cT_8$ are recorded in Table \ref{tab:combination}. Among all the errors in Table \ref{tab:combination}, the smallest one is $1.35\times 10^{-8}$, and it is the error of a weighted-average of $\CRlam $ and $\PoneAstlam$, where the weights are computed by $F_{\text{CR},\ 1}^{\text{CR}}$ and $F_{\rm P_1^\ast}^{\rm P_1}$. For the combined eigenvalue $\lambda^{\rm C,\ P_1^{\ast}}_{\rm CR, CR,\ 1}$ on $\cT_8$, the error $1.59\times 10^{-7}$  is a little larger than the smallest one in Table \ref{tab:combination}, but less computational expense is required.


\noindent \textbf{Extrapolation eigenvalues.} 
Table \ref{tab:compareextrapolation} compares the performance of the recovered eigenvalues and the extrapolation eigenvalues. It shows that the recovered eigenvalue $\lambda_{\rm CR, \ 1}^{\rm R,\ CR} $ behaves better than the extrapolation eigenvalue $\lambda^{\rm EXP}_{\rm P_1}$, but worse than $ \lambda^{\rm EXP}_{\rm CR}$.
\begin{table}[!ht]
\footnotesize
  \centering
    \begin{tabular}{c|cccccc}
    \hline
    h & $\CRlam $ &$ \Ponelam$ & $\lambda_{\rm CR}^{\rm EXP}$ & $\lambda_{\rm P_1}^{\rm EXP}$ & $\lambda_{\rm CR,\ 1}^{\rm R,\ CR}$ & $\lambda_{\rm P_1}^{\rm R,\ P_1}$ \\\hline
$\cT_3$   & -0.3407  & 3.1266  & 1.40E-02  & 8.18E-02  & 1.12E-01  & 1.0103  \\\hline
$\cT_4$   & -8.47E-02 & 7.66E-01 & 6.42E-04 & -2.04E-02 & 1.23E-03 & 6.77E-02 \\\hline
$\cT_5$  & -2.11E-02 & 1.91E-01 & 3.72E-05 & -1.34E-03 & -4.61E-04 & 4.26E-03 \\\hline
$\cT_6$  & -5.29E-03 & 4.76E-02 & 2.28E-06 & -8.24E-05 & -4.18E-05 & 2.66E-04 \\\hline
$\cT_7$  & -1.32E-03 & 1.19E-02 & 1.42E-07 & -5.12E-06 & -3.02E-06 & 1.66E-05 \\\hline
$\cT_8$ & -3.30E-04 & 2.97E-03 & 8.85E-09 & -3.19E-07 & -2.02E-07 & 1.04E-06 \\\hline
    \end{tabular}
  \caption{\footnotesize{The errors of the recovered eigenvalues and the extrapolation eigenvalues, where $\lambda_{\rm CR}^{\rm EXP}=(4\lambda_{\rm CR}^h - \lambda_{\rm CR}^{2h})/3$ and $\lambda_{\rm P_1}^{\rm EXP}=(4\lambda_{\rm P_1}^h - \lambda_{\rm P_1}^{2h})/3$ are approximate eigenvalues by extrapolation methods.}}%
  \label{tab:compareextrapolation}%
\end{table}%

\subsection{Example 2: Neumann Boundary}
Next we consider the following eigenvalue problem \eqref{variance} on the unit square $\Om=(0,1)^2$ with the following boundary conditions 
$$
u |_{x_1=0} =u |_{x_2=0}=u |_{x_2=1}=\partial_{x_1} u |_{x_1=1}=0. 
$$
In this case, there exists an eigenpair $(\lambda , u )$ where
$
\lambda =\frac{5\pi^2}{4}, u =2\cos \frac{\pi (x_1-1)}{2}\sin \pi x_2.
$

\begin{figure}[!ht]
\setlength{\abovecaptionskip}{0pt}
\setlength{\belowcaptionskip}{0pt}
\centering
\includegraphics[width=9cm,height=6cm]{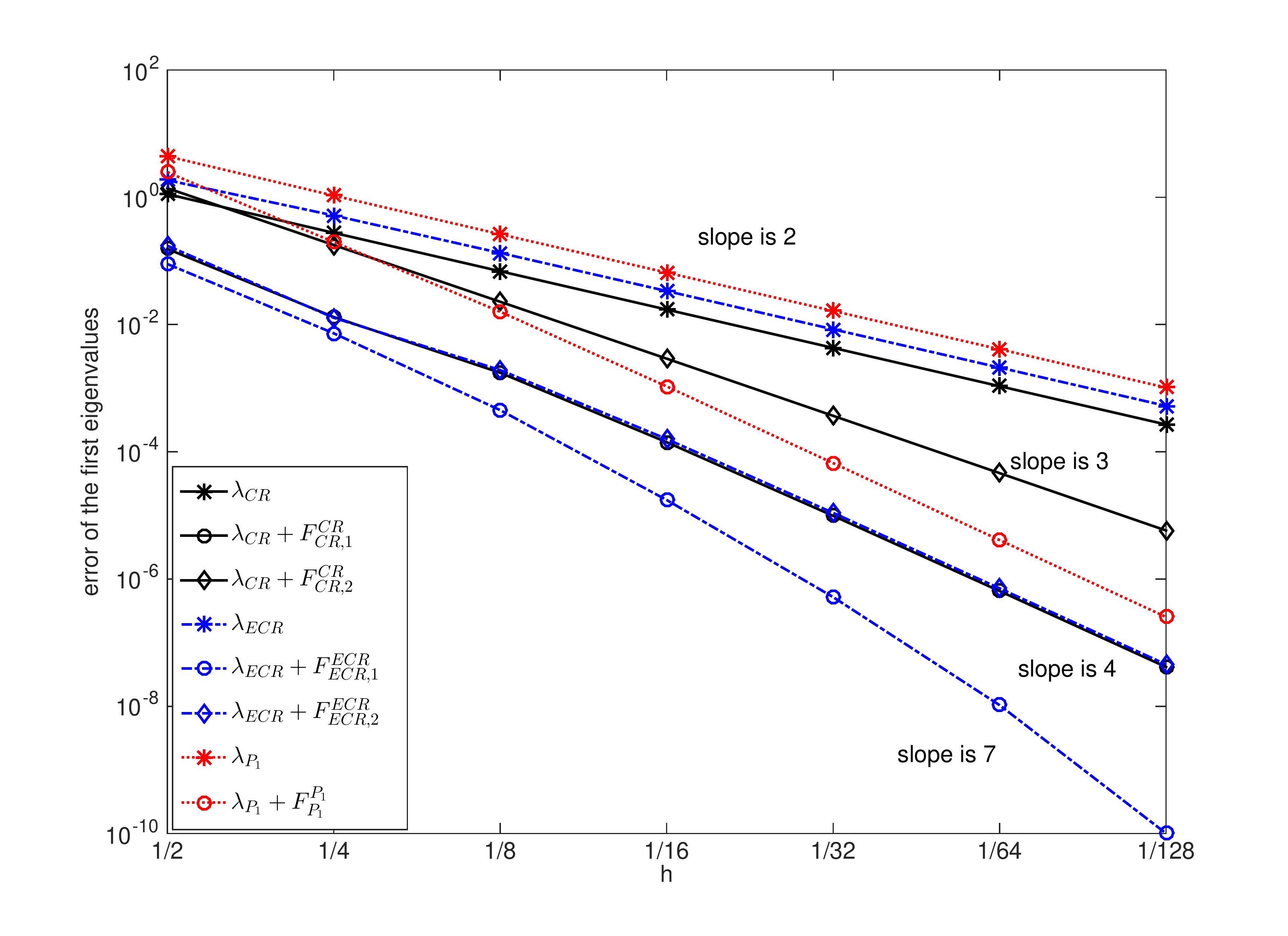}
\caption{\footnotesize{The errors of the recovered eigenvalues for Example 2.}}
\label{fig:squareNeumannCRECRP1}
\end{figure}

We solve this problem on the same sequence of uniform triangulations employed in Example 1. Figure \ref{fig:squareNeumannCRECRP1} shows that the approximate eigenvalues by the CR element, the ECR element and the conforming linear element converge at the same rate 2, the recovered eigenvalue $\lambda_{\rm CR,\ 2}^{\rm R,\ CR}$ converges at a rate 3, the recovered eigenvalues  $\lambda_{\rm CR,\ 1}^{\rm R,\ CR}$, $\lambda_{\rm ECR,\ 2}^{\rm R,\ ECR}$ and $\lambda_{\rm P_1}^{\rm R,\ P_1}$ converge  at a rate 4. Especially, the recovered eigenvalue $\lambda_{\rm ECR,\ 1}^{\rm R,\ ECR}$ converges at a strikingly higher rate 7.

\subsection{Example 3: Triangle Domain}
In this experiment, we consider the eigenvalue problem \eqref{variance} on the domain which is an equilateral triangle:
$$
\Om=\big \{(x_1,x_2)\in \mathbb{R}^2:0< x_2< \sqrt{3}x_1, \sqrt{3}(1-x_1)<x_2\big \}.
$$
The boundary consists of three parts:
$
\Gamma_1=\big\{(x_1,x_2)\in \mathbb{R}^2: x_2=\sqrt{3}x_1,\ 0.5\le x_1\le 1 \big\},
$
$
\Gamma_2=\big\{(x_1,x_2)\in \mathbb{R}^2: x_2=\sqrt{3}(1-x_1),\ 0.5\le x_1\le 1 \big\},
$
$
\Gamma_3=\big\{(x_1,x_2)\in \mathbb{R}^2: x_1=1 ,\ 0\le x_2\le 1 \big\}.
$
Under the boundary conditions $u |_{\Gamma_1\cup \Gamma_2}=0$ and $\partial_{x_1} u |_{\Gamma_3}=0$,
there exists an eigenpair $(\lambda  , u )$, where $\lambda =\frac{16\pi^2}{3}$ and
$$
u =\frac{2\sqrt[4]{12}}{3}\big ( \sin \frac{4\pi x_2}{\sqrt{3}} + \sin 2\pi(x_1-\frac{x_2}{\sqrt{3}}) + \sin 2\pi(1-x_1-\frac{x_2}{\sqrt{3}})\big ).
$$
\begin{table}[!ht]
\footnotesize
  \centering
    \begin{tabular}{c|cccccc}
    \hline
          &$\cT_1$ & $\cT_2$ & $\cT_3$ &$\cT_4$& $\cT_5$   \\\hline
    $\lambda -\CRlam $&  3.6697 & 9.61E-01 & 2.43E-01 & 6.10E-02 & 1.53E-02 \\
    $\lambda -\lambda_{\rm CR,\ 1}^{\rm R,\ CR}$ &-2.10E-01 &-9.46E-03 &-2.80E-04 &-6.08E-06 &-8.80E-09 \\
    $\lambda -\lambda_{\rm CR,\ 2}^{\rm R,\ CR}$ &	6.67E-01 &1.11E-01& 1.46E-02	 &1.81E-03 &2.24E-04 \\\hline
    $\lambda -\ECRlam $&  4.8974 & 1.3186 & 3.36E-01 & 8.44E-02 & 2.11E-02  \\
    $\lambda -\lambda_{\rm ECR,\ 1}^{\rm R,\ ECR}$& 0.3463 & 2.85E-02 & 2.27E-03 & 1.55E-04 & 1.01E-05  \\
    $\lambda -\lambda_{\rm ECR,\ 2}^{\rm R,\ ECR}$&2.12E-03 &1.25E-04 & 3.48E-04 &3.28E-05 &2.40E-06 \\\hline
    $\lambda -\Ponelam$& -11.1936 & -2.879 & -7.29E-01 & -1.83E-01 & -4.58E-02\\
    $\lambda -\lambda_{\rm P_1}^{\rm R,\ P_1}$&1.0209 & 0.1005 & 3.49E-03 & 9.62E-05 & 2.21E-06  \\\hline
    \end{tabular}%
  \caption{\footnotesize The errors and convergence rates of the recovered eigenvalues for Example 3.}
    \label{tab:triangleCRECRP1}%
\end{table}%
Refine the domain $\Om$ into four half-sized triangles twice and denote the resulted triangulation by $\cT_1$. Each triangulation $\cT_i$ is refined into a half-sized triangulation uniformly, to get a higher level triangulation $\cT_{i+1}$. It is showed in Table \ref{tab:triangleCRECRP1} that the recovered eigenvalues converge much faster than discrete eigenvalues and the recovered eigenvalue $\lambda_{\rm CR,\ 1}^{\rm R,\ CR}$ on $\cT_5$ achieves the smallest error $8.80\times 10^{-9}$.

\subsection{Example 4: L-shaped Domain}
Next we consider the eigenvalue problem \eqref{variance} 
on a L-shaped domain $\Om=(-1,1)^2/[0,1]\times[-1,0]$. For this problem, the third and the eighth eigenvalues are known to be $2\pi^2$ and $4\pi^2$, respectively, and the corresponding eigenfunctions are smooth.

In the computation, the level one triangulation is obtained by dividing the domain into three unit squares, each of which is further divided into two triangles. Each triangulation is refined into a half-sized triangulation uniformly to get a higher level triangulation. Since the exact eigenvalues of this problem are unknown, we solve the first eight eigenvalues by the conforming $\rm P_3$ element on the mesh $\cT_9$, and take them as reference eigenvalues.

\begin{center}
\setlength{\tabcolsep}{1.2mm}{
\begin{table}[!ht]
\scriptsize
\renewcommand\arraystretch{1.5}
    \begin{tabular}{c|cccccccc}
    \hline
          &$\lambda_1 $ & $\lambda_2 $ & $\lambda_3 $ & $\lambda_4 $ & $\lambda_5 $ & $\lambda_6 $ & $\lambda_7 $ & $\lambda_8 $  \\\hline
    $\CRlam$ & 9.68E-04 & 9.65E-05 & 6.69E-05 & 1.79E-04 & 8.75E-04 & 7.75E-04 & 4.44E-04 & 3.48E-04 \\\hline
    $\Ponelam$ & -1.10E-03 & -4.52E-04 & -6.02E-04 & -8.83E-04 & -1.37E-03 & -1.37E-03 & -1.21E-03 & -1.23E-03 \\\hline
    $\PoneAstlam$ &  -1.41E-03 & -4.53E-04 & -6.03E-04 & -8.85E-04 & -1.60E-03 & -1.51E-03 & -1.21E-03 & -1.32E-03  \\\hline
    $\lambda_{\rm CR,\ 1}^{\rm R,\ CR}$ & 4.22E-04 & 8.29E-07 & -1.60E-07 & -2.18E-07 & 3.11E-04 & 1.80E-04 & 1.11E-06 & -8.53E-07 \\\hline
    $\lambda_{\rm CR,\ 2}^{\rm R,\ CR}$ & 2.41E-05 & 1.04E-05 & 8.32E-06 & 9.85E-06 & 2.48E-05 & 1.68E-05 & 1.70E-05 & 1.30E-05\\\hline
    $\lambda_{\rm P_1}^{\rm R,\ P_1}$ & 4.78E-04 & 7.24E-06 & 8.03E-07 & 3.35E-06 & 3.54E-04 & 2.09E-04 & 1.72E-05 & 4.11E-06 \\\hline
    $\lambda_{\rm P_1^{\ast}}^{\rm R,\ P_1^\ast}$ &  7.53E-04 & 7.93E-06 & 2.46E-06 & 5.36E-06 & 5.62E-04 & 3.28E-04 & 1.84E-05 & 8.34E-06\\ \hline
    $\lambda^{\rm C,\ P_1}_{\rm CR, P_1, 2}$ & 1.94E-04 & 9.88E-06 & 7.65E-06 & 8.80E-06 & 1.34E-04 & 7.89E-05 & 1.70E-05 & 1.11E-05\\ \hline
    $\lambda^{\rm C,\ P_1^{\ast}}_{\rm CR, P_1^{\ast}, 2}$ &2.46E-04 & 9.99E-06 & 7.80E-06 & 9.13E-06 & 1.76E-04 & 1.08E-04 & 1.73E-05 & 1.21E-05  \\ \hline
    \end{tabular}%
  \caption{\footnotesize The relative errors of different approximations to the first eight eigenvalue on $\cT_8$ for Example 4.}
  \label{tab:Lshape}%
\end{table}}
\end{center}

Table \ref{tab:Lshape} compares the relative errors of the first eight approximate eigenvalues on $\cT_8$ by different methods. It implies that the error of the recovered eigenvalue $\lambda_{\rm CR,\ 1}^{\rm R,\ CR}$ is slightly smaller than that of $\lambda_{\rm P_1}^{\rm R,\ P_1}$. For the first eigenvalue, the eigenfunction is singular, and the corresponding recovered eigenvalue $\lambda_{\rm CR,\ 2}^{\rm R,\ CR}$ achieves higher accuracy than the recovered eigenvalues $\lambda_{\rm CR,\ 1}^{\rm R,\ CR}$ and $\lambda_{\rm P_1}^{\rm R,\ P_1}$. For the third eigenvalue, the corresponding eigenfunction is smooth, and the relative errors of the third recovered eigenvalues $\lambda_{\rm CR,\ 1}^{\rm R,\ CR}$ and $\lambda_{\rm P_1}^{\rm R,\ P_1}$ are $1.60\times 10^{-7}$ and $8.03\times 10^{-7}$, respectively, smaller than that of the recovered eigenvalue $\lambda_{\rm CR,\ 2}^{\rm R,\ CR}$. 

Note that the corresponding eigenfunctions to the third and the eighth eigenvalues are smooth, and the other eigenfunctions are singular. 
Note that the relative error for recovered eigenvalue $\lambda_{\rm CR,\ 2}^{\rm R,\ CR}$ differs little from each other.  However, for $\lambda_{\rm CR,\ 1}^{\rm R,\ CR}$, $\lambda_{\rm P_1}^{\rm R,\ P_1}$ and  $\lambda_{\rm P_1^{\ast}}^{\rm R,\ P_1^\ast}$, the relative errors differ a lot for different eigenvalues. 
This fact implies that the accuracy of the second type error estimates is insensitive to the regularity of the corresponding eigenfunctions, while the accuracy of the first type error estimates relies heavily on the regularity of eigenfunctions. 

\begin{figure}[!ht]
\setlength{\abovecaptionskip}{0pt}
\setlength{\belowcaptionskip}{0pt}
\centering
\includegraphics[width=5cm,height=4.5cm]{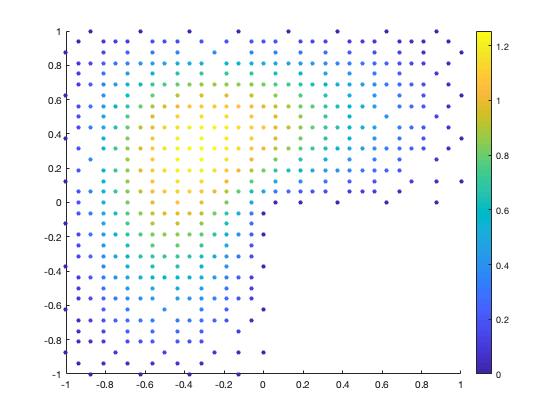}
\includegraphics[width=5cm,height=4.5cm]{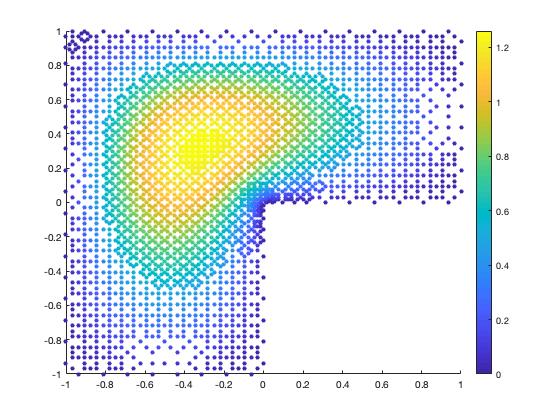}
\includegraphics[width=5cm,height=4.5cm]{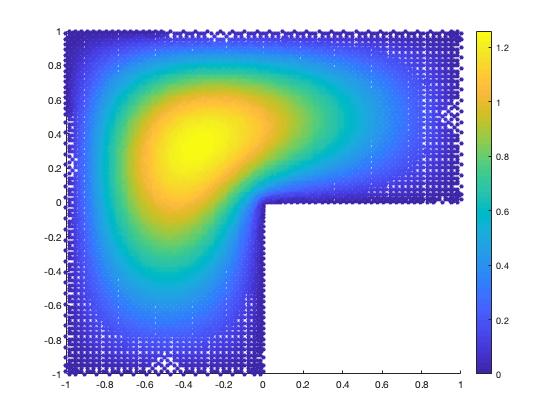}
\includegraphics[width=5cm,height=4.5cm]{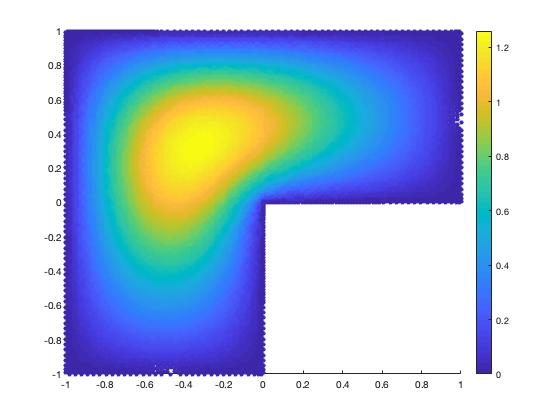}
\caption{\footnotesize{ The first eigenfunctions on adaptive triangulations with $k=5$, $k=10$, $k=15$   and $k=20$  for Example 4.}}
\label{fig:Lmesh}
\end{figure}

\begin{figure}[!ht]
\setlength{\abovecaptionskip}{0pt}
\setlength{\belowcaptionskip}{0pt}
\centering
\includegraphics[width=9cm,height=6cm]{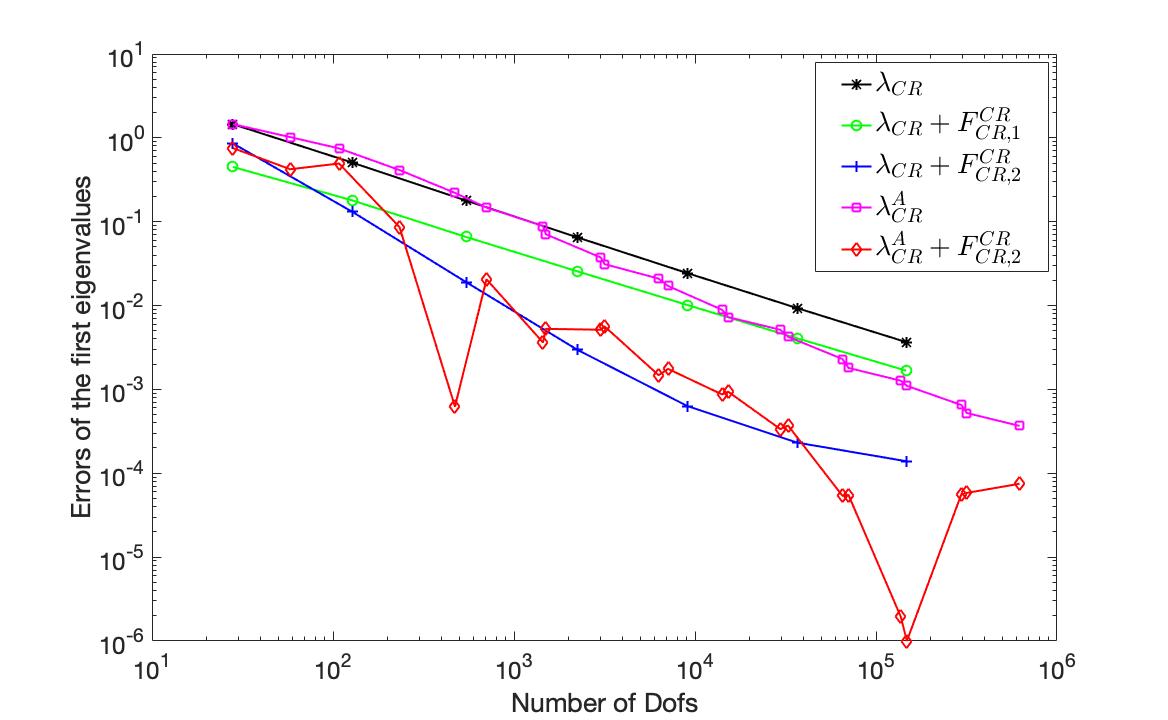}
\caption{\footnotesize{The errors of the first eigenvalues for Example 4.}}
\label{fig:LAFEM}
\end{figure}

Figure \ref{fig:Lmesh} plots the resulted eigenfunctions on triangulations $\cT_{5}$, $\cT_{10}$, $\cT_{15}$ and $\cT_{20}$  from the adaptive algorithm in Section 5.1. 
Figure \ref{fig:LAFEM} plots the relationship between the errors of the first discrete eigenvalues and the sizes of the discrete eigenvalue problems. It shows that the accuracy of the approximate eigenvalues $\CRlam^{\rm A}$ on adaptive triangulations increases faster than that of $\lambda_{\rm CR,\ 1}^{\rm R,\ CR}$  on uniform triangulations. The second type a posteriori error estimates on the uniform triangulations are more accurate than those on adaptive triangulations. 


\begin{figure}[!ht]
\begin{center}
\begin{tikzpicture}[xscale=1,yscale=1]
\draw[-] (-2,-2) -- (-1,-2);
\draw[-] (1,-2) -- (2,-2);
\draw[-] (-2,2) -- (-1,2);
\draw[-] (1,2) -- (2,2);
\draw[-] (1,-2) -- (1,2);
\draw[-] (2,-2) -- (2,2);
\draw[-] (-1,-2) -- (-1,2);
\draw[-] (-2,-2) -- (-2,2);
\draw[-] (-2,0.5) -- (2,0.5);
\draw[-] (-2,-0.5) -- (2,-0.5);

\node[below] at (-2,-2) {\footnotesize (-2,-2)};
\node[above] at (-2,2) {\footnotesize (-2,2)};
\node[below] at (2,-2) {\footnotesize (2,-2)};
\node[above] at (2,2) {\footnotesize (2,2)};

\draw[-] (-2, 0) -- (2,0);
\draw[-] (-1,-0.5) -- (-1,0.5);
\draw[-] (-0.5,-0.5) -- (-0.5,0.5);
\draw[-] (0,-0.5) -- (0,0.5);
\draw[-] (0.5,-0.5) -- (0.5,0.5);
\draw[-] (1,-0.5) -- (1,0.5);

\draw[-] (-2,-0.5) -- (-1,0.5);
\draw[-] (-2,-1) -- (-0.5,0.5);
\draw[-] (-1,-0.5) -- (0,0.5);
\draw[-] (-0.5,-0.5) -- (0.5,0.5);
\draw[-] (0,-0.5) -- (1,0.5);
\draw[-] (0.5,-0.5) -- (1.5,0.5);
\draw[-] (1,-0.5) -- (2,0.5);
\draw[-] (1,-1) -- (2,0);

\draw[-] (-2,1.5) -- (-1.5,2);
\draw[-] (1,1.5) -- (1.5,2);
\draw[-] (-1,-1.5) -- (-1.5,-2);
\draw[-] (2,-1.5) -- (1.5,-2);

\draw[-] (-2,1.5) -- (-1,1.5);
\draw[-] (-2,1) -- (-1,1);
\draw[-] (-2,-1) -- (-1,-1);
\draw[-] (-2,-1.5) -- (-1,-1.5);
\draw[-] (2,1.5) -- (1,1.5);
\draw[-] (2,1) -- (1,1);
\draw[-] (2,-1) -- (1,-1);
\draw[-] (2,-1.5) -- (1,-1.5);
\draw[-] (-1.5,-2) -- (-1.5,2);
\draw[-] (1.5,-2) -- (1.5,2);

\draw[-] (-2,0) -- (-1,1);
\draw[-] (-2,0.5) -- (-1,1.5);
\draw[-] (-2,1) -- (-1,2);
\draw[-] (-2,-2) -- (-1,-1);
\draw[-] (-2,-1.5) -- (-1,-0.5);

\draw[-] (1,0) -- (2,1);
\draw[-] (1,0.5) -- (2,1.5);
\draw[-] (1,1) -- (2,2);
\draw[-] (1,1) -- (2,2);
\draw[-] (1,-1.5) -- (2,-0.5);
\draw[-] (1,-2) -- (2,-1);
\end{tikzpicture}
\end{center}
\caption{\footnotesize Initial triangulation for Example 5.}
\label{fig:Hshape}
\end{figure}
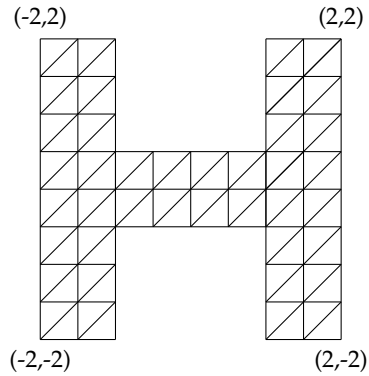

\subsection{Example 5: H-shaped Domain}
Next we consider the  eigenvalue problem \eqref{variance} on a H-shaped domain in Figure \ref{fig:Hshape}.  Note that the first eigenfunction is singular.

The level one triangulation is uniform and shown in Figure \ref{fig:Hshape}. Each triangulation is refined into a half-sized triangulation uniformly to get a higher level triangulation. Since the exact eigenvalues of this problem are unknown, the conforming $\rm P_3$ element is employed to solve the eigenvalue problem on the mesh $\cT_9$, and the resulted approximate eigenvalues are taken as reference eigenvalues.

\begin{table}[!ht]
\footnotesize
  \centering
    \begin{tabular}{c|ccccccc}
\hline
   h     &$\CRlam$&$\Ponelam$&$ \lambda_{\rm CR,\ 1}^{\rm R,\ CR} $&$ \lambda_{\rm CR,\ 2}^{\rm R,\ CR} $&$ \lambda_{\rm P_1}^{\rm R,\ P_1} $&$ \lambda_{\rm P_1^\ast}^{\rm R,\ P_1^\ast} $&$ \lambda_{\rm CR,\ CR, 2}^{\rm C,\ P_1^\ast}$ \\\hline
$ \cT_1  $& 1.61E+00 & -2.84E+00 & 6.27E-01 & 4.39E-01 & 1.93E+00 & 2.01E+00 & 6.46E-01 \\
$ \cT_2  $& 5.74E-01 & -7.99E-01 & 2.05E-01 & 6.03E-02 & 3.37E-01 & 4.70E-01 & 1.21E-01 \\
$ \cT_3  $& 2.06E-01 & -2.54E-01 & 7.52E-02 & 8.46E-03 & 8.76E-02 & 1.44E-01 & 3.07E-02 \\
$ \cT_4  $& 7.54E-02 & -8.68E-02 & 2.91E-02 & 9.72E-04 & 3.04E-02 & 5.26E-02 & 1.00E-02 \\
$ \cT_5  $& 2.83E-02 & -3.10E-02 & 1.15E-02 & 8.95E-05 & 1.17E-02 & 2.05E-02 & 3.74E-03 \\
$ \cT_6  $& 1.09E-02 & -1.15E-02 & 4.56E-03 & 2.06E-05 & 4.61E-03 & 8.10E-03 & 1.48E-03 \\\hline
    \end{tabular}%
  \caption{\footnotesize The errors of the first eigenvalue approximations by different methods for Example 5.}
  \label{tab:Hshape}%
\end{table}%

\begin{figure}[!ht]
\setlength{\abovecaptionskip}{0pt}
\setlength{\belowcaptionskip}{0pt}
\centering
\includegraphics[width=5cm,height=4.5cm]{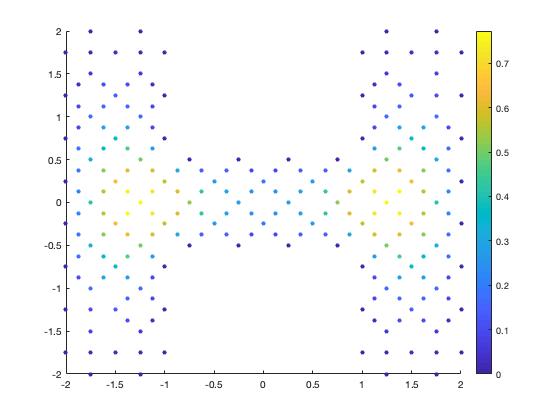}
\includegraphics[width=5cm,height=4.5cm]{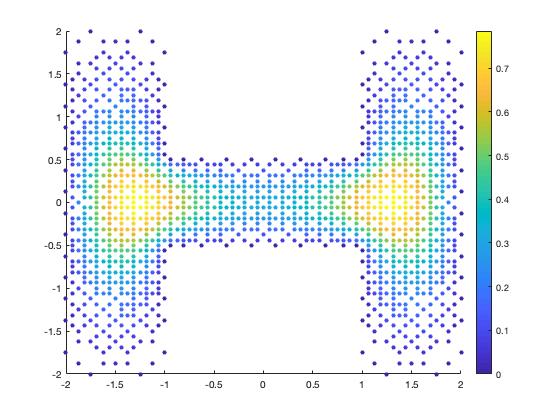}
\includegraphics[width=5cm,height=4.5cm]{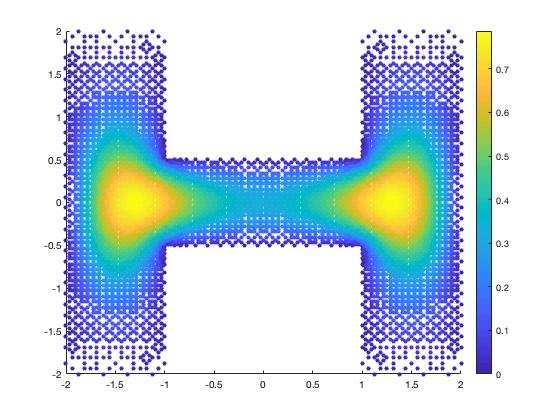}
\includegraphics[width=5cm,height=4.5cm]{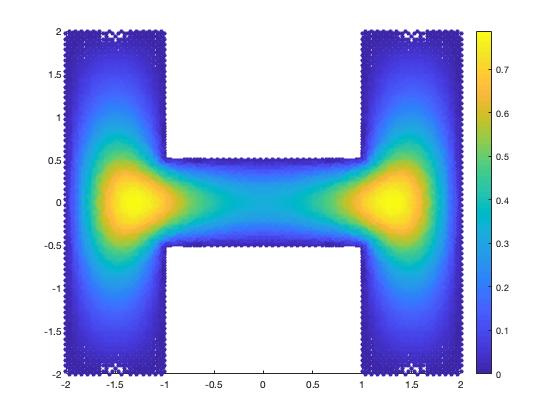}
\caption{\footnotesize{ The first eigenfunctions on adaptive triangulations with $k=2$, $k=5$, $k=10$   and $k=15$  for Example 5.}}
\label{fig:Hmesh}
\end{figure}

\begin{figure}[!ht]
\setlength{\abovecaptionskip}{0pt}
\setlength{\belowcaptionskip}{0pt}
\centering
\includegraphics[width=9cm,height=6cm]{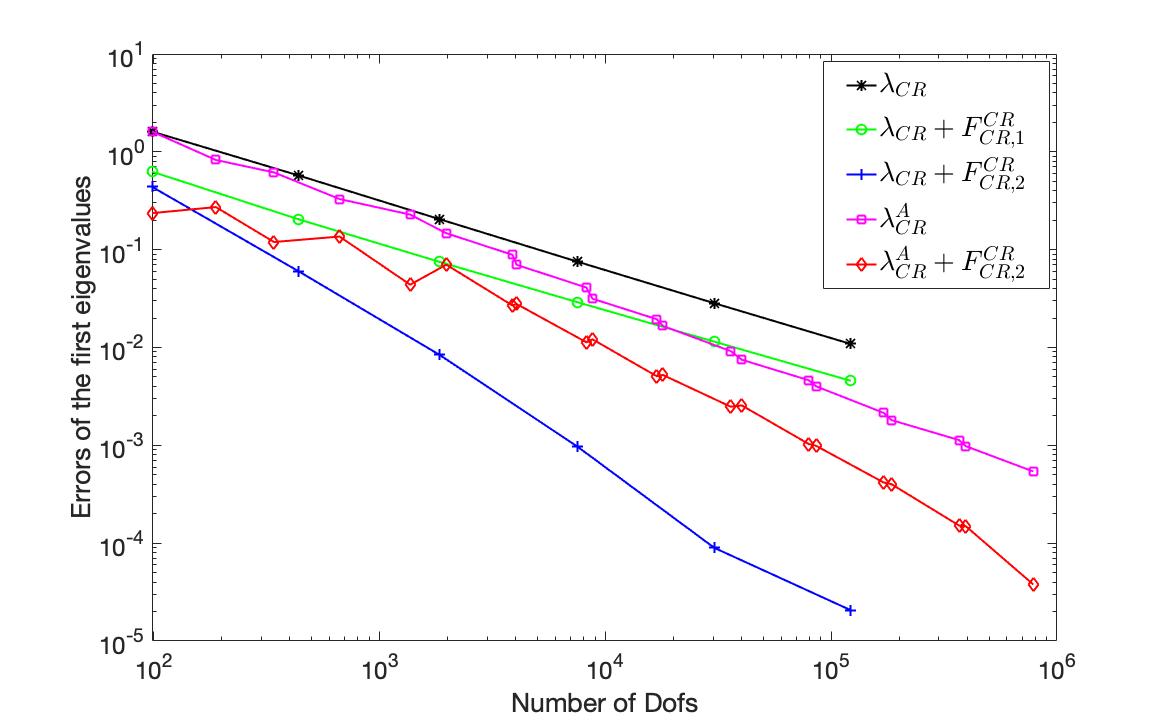}
\caption{\footnotesize{The errors of the first eigenvalue for Example 5.}}
\label{fig:HAFEM}
\end{figure}

As presented in Table \ref{tab:Hshape}, the recovered eigenvalue $ \lambda_{\rm CR,\ 2}^{\rm R,\ CR} $ performs  better than $ \lambda_{\rm CR,\ 1}^{\rm R,\ CR} $ when the eigenfunctions are singular. The errors of $ \lambda_{\rm CR,\ 1}^{\rm R,\ CR} $ and $ \lambda_{\rm CR,\ 2}^{\rm R,\ CR} $ on $\cT_6$ are  $4.56\times 10^{-3}$ and $2.06\times 10^{-5}$, respectively. The recovered eigenvalue $ \lambda_{\rm CR,\ 2}^{\rm R,\ CR} $ on $\cT_6$ achieves the smallest error, remarkably smaller than the combined eigenvalue with error $1.48\times 10^{-3}$. The reason why the combined eigenvalues behave worse is that the asymptotically exact a posteriori error estimates $F_{\rm P_1^\ast}^{\rm CR}$  are less accurate than $F_{\rm CR,\ 2}^{\rm CR}$ when the eigenfunctions are singular. 

Figure \ref{fig:Hmesh} plots the resulted eigenfunctions on triangulations $\cT_{2}$, $\cT_{5}$, $\cT_{10}$ and $\cT_{15}$  from the adaptive algorithm in Section 5.1. 
Figure \ref{fig:HAFEM} plots the relationship between errors of the first eigenvalues and the sizes of discrete eigenvalue problems. The performance of these approximations is quite similar to that in Example 4. 
Compared to the first type recovered eigenvalues  $\lambda_{\rm CR,\ 1}^{\rm R,\ CR}$ on uniform triangulations, approximate eigenvalue $\CRlam^{\rm A}$ on adaptive triangulations achieves higher accuracy when the size of the discrete problem is large enough. 
The recovered eigenvalue $\lambda_{\rm CR,\ 2}^{\rm R,\ CR}$ on the uniform triangulations behaves much better than other approximations.

\subsection{Example 6: Hollow-shaped Domain}
Consider the eigenvalue problem \eqref{variance} with the boundary condition
$$
u|_{\Gamma_D}=0 \text{ and } \partial_{x_1} u |_{\Gamma_N}=0
$$
on a hollow-shaped domain in Figure \ref{fig:Hollowshape} with $\Gamma_N=\{(x_1, x_2): x_1=-1, -1\leq x_2\leq 1\}$ and $\partial \Omega=\Gamma_D \cup \Gamma_N$. Note that the first eigenfunction is singular.
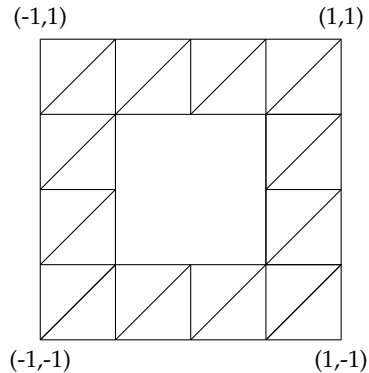
\begin{figure}[!ht]
\begin{center}
\begin{tikzpicture}[xscale=2,yscale=2]
\draw[-] (-1,-1) -- (1,-1);
\draw[-] (1,-1) -- (1,1);
\draw[-] (-1,-1) -- (-1,1);
\draw[-] (-1,1) -- (1,1);
\draw[-] (-0.5,-1) -- (-0.5,1);
\draw[-] (0.5,-1) -- (0.5,1);

\draw[-] (-1,-0.5) -- (1,-0.5);
\draw[-] (-1,0.5) -- (1,0.5);
\node[below] at (-1,-1) {\footnotesize (-1,-1)};
\node[above] at (-1,1) {\footnotesize (-1,1)};
\node[below] at (1,-1) {\footnotesize (1,-1)};
\node[above] at (1,1) {\footnotesize (1,1)};

\draw[-] (-1,0.5) -- (-0.5,1);
\draw[-] (-1,0) -- (0,1);\draw[-] (-1,0) -- (-0.5,0);
\draw[-] (-1,-0.5) -- (-0.5,0);
\draw[-] (-1,-1) -- (-0.5,-0.5);\draw[-] (-1,-1) -- (-0.5,-0.5);
\draw[-] (-0.5,-0.5) -- (-0.5,0.5);

\draw[-] (0.5,0.5) -- (1,1);\draw[-] (0.5,0.5) -- (1,0.5);
\draw[-] (0.5,0) -- (1,0.5);\draw[-] (0.5,0) -- (1,0);
\draw[-] (0,-1) -- (1,0);\draw[-] (0.5,-0.5) -- (1,-0.5);
\draw[-] (0.5,-1) -- (1,-0.5);\draw[-] (0.5,-1) -- (1,-0.5);
\draw[-] (0.5,-0.5) -- (0.5,0.5);

\draw[-] (0,0.5) -- (0,1);\draw[-] (0,-0.5) -- (0,-1);
\draw[-] (0,0.5) -- (0.5,1);\draw[-] (0,-0.5) -- (-0.5,-1);
\end{tikzpicture}
\end{center}
\caption{\footnotesize Initial triangulation for Example 6.}
\label{fig:Hollowshape}
\end{figure}

The level one triangulation is uniform and shown in Figure \ref{fig:Hollowshape}. Each triangulation is refined into a half-sized triangulation uniformly to get a higher level triangulation. Since the exact eigenvalues of this problem are unknown, the conforming $\rm P_3$ element is employed to solve the eigenvalue problem on the mesh $\cT_9$, and the resulted approximate eigenvalues are taken as reference eigenvalues.
\begin{table}[!ht]
\footnotesize
  \centering
   \begin{tabular}{c|ccccccc}
\hline
   h     &$\CRlam$&$\Ponelam$&$ \lambda_{\rm CR,\ 1}^{\rm R,\ CR} $&$ \lambda_{\rm CR,\ 2}^{\rm R,\ CR} $&$ \lambda_{\rm P_1}^{\rm R,\ P_1} $&$ \lambda_{\rm P_1^\ast}^{\rm R,\ P_1^\ast} $&$ \lambda_{\rm CR,\ P_1, 2}^{\rm C,\ P_1} $ \\\hline
$ \cT_2  $& 6.42E-01 & -1.32E+00 & 1.28E-01 & 1.14E-01 & 7.16E-01 & 8.76E-01 & 2.38E-01 \\
$ \cT_3  $& 2.23E-01 & -3.84E-01 & 6.93E-02 & 2.08E-02 & 1.24E-01 & 1.80E-01 & 5.02E-02 \\
$ \cT_4  $& 7.99E-02 & -1.19E-01 & 2.81E-02 & 2.74E-03 & 3.44E-02 & 5.64E-02 & 1.33E-02 \\
$ \cT_5  $& 2.94E-02 & -3.90E-02 & 1.12E-02 & 3.77E-04 & 1.21E-02 & 2.08E-02 & 4.63E-03 \\
$ \cT_6  $& 1.11E-02 & -1.34E-02 & 4.53E-03 & 1.14E-04 & 4.68E-03 & 8.12E-03 & 1.85E-03 \\
$ \cT_7  $& 4.32E-03 & -4.72E-03 & 1.87E-03 & 9.75E-05 & 1.90E-03 & 3.26E-03 & 7.99E-04 \\\hline
    \end{tabular}%
  \caption{\footnotesize The errors of the first eigenvalue approximations by different methods for Example 6.}
  \label{tab:Huishape}%
\end{table}%

\begin{figure}[!ht]
\setlength{\abovecaptionskip}{0pt}
\setlength{\belowcaptionskip}{0pt}
\centering
\includegraphics[width=8cm,height=6cm]{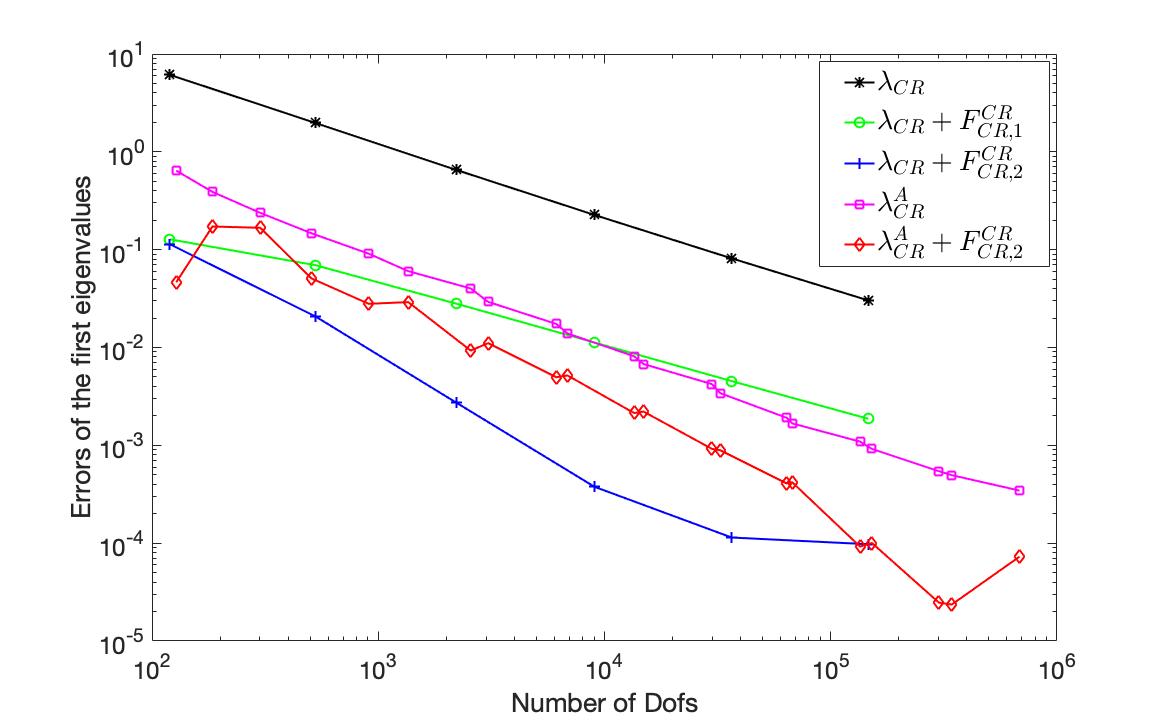}
\caption{\footnotesize{The errors of the first eigenvalues for Example 6.}}
\label{fig:HUIAFEM}
\end{figure}

\begin{figure}[!ht]
\setlength{\abovecaptionskip}{0pt}
\setlength{\belowcaptionskip}{0pt}
\centering
\includegraphics[width=5cm,height=4.5cm]{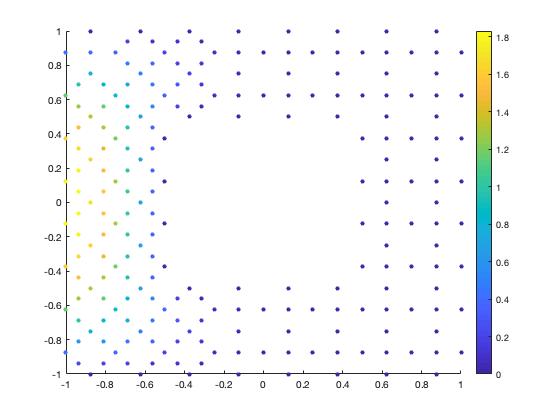}
\includegraphics[width=5cm,height=4.5cm]{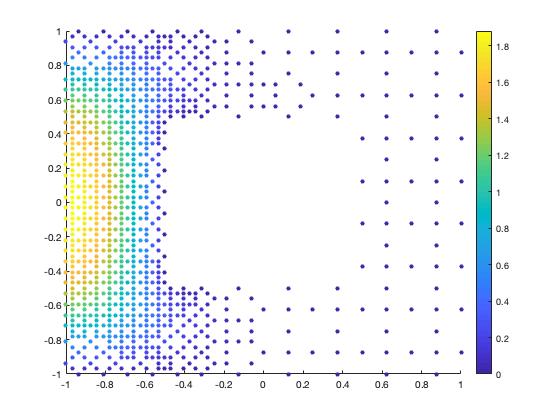}
\includegraphics[width=5cm,height=4.5cm]{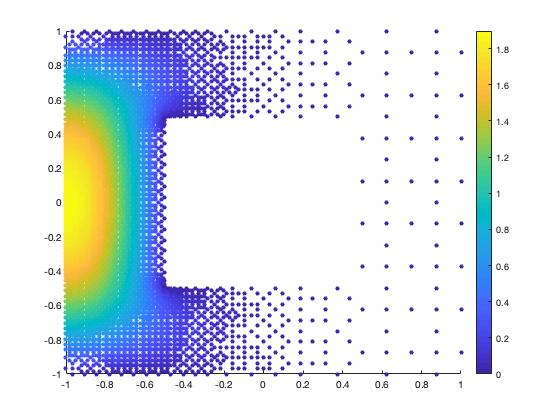}
\includegraphics[width=5cm,height=4.5cm]{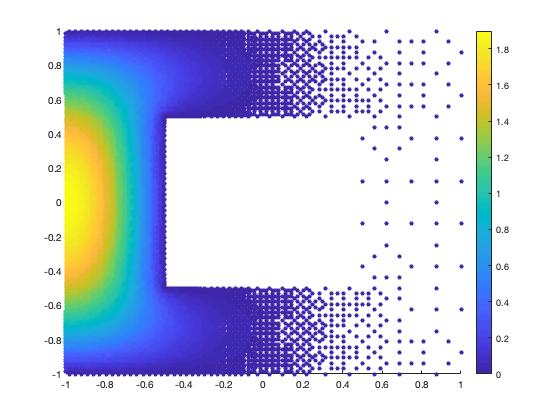}
\caption{\footnotesize{ The first eigenfunctions on adaptive triangulations with $k=2$, $k=5$, $k=10$   and $k=15$  for Example 6.}}
\label{fig:HUImesh}
\end{figure}

Table \ref{tab:Huishape} lists the errors of the approximate eigenvalues on each mesh by different methods. The recovered eigenvalue  $ \lambda_{\rm CR,\ 2}^{\rm R,\ CR} $ performs remarkably better than all the other recovered eigenvalues and combined eigenvalues when the eigenfunctions are singular. 
The error of the recovered eigenvalue $  \lambda_{\rm CR,\ P_1, 2}^{\rm C,\ P_1} $  on $\cT_7$ is $7.99\times 10^{-4}$, it is smaller than all the other approximations but $ \lambda_{\rm CR,\ 2}^{\rm R,\ CR} $. 
It implies that if the corresponding eigenfunctions are  not smooth enough,  the second type of asymptotically exact a posteriori error estimate $F_{\rm CR,\ 2}^{\rm CR}$ admits higher accuracy than $F_{\rm CR,\ 1}^{\rm CR}$.

Figure \ref{fig:HUIAFEM} plots the relationship between errors of the first eigenvalues and the sizes of discrete eigenvalue problems. Figure \ref{fig:HUImesh} plots the resulted eigenfunctions on triangulations $\cT_{2}$, $\cT_{5}$, $\cT_{10}$ and $\cT_{15}$  from the adaptive algorithm in Section 5.1.  
The performance of these approximations is quite similar to those in Example 4 and Example 5. The second type of recovered eigenvalues on uniform triangulations achieves remarkable higher accuracy  than most of the other approximations.

\bibliographystyle{plain}
\bibliography{bibifile}

\end{document}